\documentclass[12pt]{amsart}
\usepackage{amscd,amsmath,amsthm,amssymb,graphics}

\usepackage{amsmath}
\usepackage{amssymb}
\usepackage{amsbsy}
\usepackage{graphicx}
\usepackage{amsfonts}

\usepackage{amscd,amsmath,amsthm,amssymb}
\usepackage{pstricks,pst-plot,pst-3d}
\usepackage{lmodern,pst-node}

\usepackage{pgf,tikz}
\usetikzlibrary{arrows}

\def\NZQ{\Bbb}               

\def\ZZ{{\NZQ Z}}
\def\RR{{\NZQ R}}

\def\frk{\frak}               

\def\Phi{{\frk n}}
\def\Phi{{\frk N}}
\def\wb{{\bold w}}
\def\ub{{\bold u}}

\def\vb{{\bold v}}
\def\opn#1#2{\def#1{\operatorname{#2}}} 
\opn\chara{char} \opn\length{\ell} \opn\pd{pd} \opn\rk{rk}
\opn\projdim{proj\,dim} \opn\injdim{inj\,dim} \opn\rank{rank}
\opn\depth{depth} \opn\grade{grade} \opn\height{height}
\opn\embdim{emb\,dim} \opn\codim{codim}

\opn\Tr{Tr} \opn\bigrank{big\,rank}
\opn\superheight{superheight}\opn\lcm{lcm}
\opn\trdeg{tr\,deg}
\opn\reg{reg} \opn\lreg{lreg} \opn\ini{in} \opn\lpd{lpd}
\opn\size{size}\opn\bigsize{bigsize}
\opn\cosize{cosize}\opn\bigcosize{bigcosize}
\opn\sdepth{sdepth}\opn\sreg{sreg}
\opn\link{link}\opn\fdepth{fdepth}
\opn\index{index}
\opn\index{index}
\opn\indeg{indeg}
\opn\N{N}
\opn\SSC{SSC}
\opn\SC{SC}
\opn\conv{conv}
\opn\div{div} \opn\Div{Div} \opn\cl{cl} \opn\Cl{Cl}
\opn\Spec{Spec} \opn\Supp{Supp} \opn\supp{supp} \opn\Sing{Sing}
\opn\Ass{Ass} \opn\Min{Min}\opn\Mon{Mon} \opn\dstab{dstab} \opn\astab{astab}
\opn\Syz{Syz}
\opn\reg{reg}
\opn\Ann{Ann} \opn\Rad{Rad} \opn\Soc{Soc}
\opn\Im{Im} \opn\Ker{Ker} \opn\Coker{Coker} \opn\Am{Am}
\opn\Hom{Hom} \opn\Tor{Tor} \opn\Ext{Ext} \opn\End{End}
\opn\Aut{Aut} \opn\id{id}

\opn\nat{nat}
\opn\pff{pf}
\opn\Pf{Pf} \opn\GL{GL} \opn\SL{SL} \opn\mod{mod} \opn\ord{ord}
\opn\Gin{Gin} \opn\Hilb{Hilb}\opn\sort{sort}
\opn\initial{init}
\opn\ende{end}
\opn\height{height}
\opn\type{type}
\opn\aff{aff} \opn\con{conv} \opn\relint{relint} \opn\st{st}
\opn\lk{lk} \opn\cn{cn} \opn\core{core} \opn\vol{vol}
\opn\link{link} \opn\star{star}\opn\lex{lex}\opn\Mon{Mon}\opn\Min{Min}
\opn\gr{gr}

\def\pot#1#2{#1[\kern-0.28ex[#2]\kern-0.28ex]}

\opn\dirlim{\underrightarrow{\lim}}
\opn\inivlim{\underleftarrow{\lim}}
\let\to=\rightarrow

\def\Implies{\ifmmode\Longrightarrow \else
        \unskip${}\Longrightarrow{}$\ignorespaces\fi}
\def\implies{\ifmmode\Rightarrow \else
        \unskip${}\Rightarrow{}$\ignorespaces\fi}
\def\iff{\ifmmode\Longleftrightarrow \else
        \unskip${}\Longleftrightarrow{}$\ignorespaces\fi}

\let\:=\colon
\newtheorem{Theorem}{Theorem}[section]
 \newtheorem{Lemma}[Theorem]{Lemma}
 \newtheorem{Corollary}[Theorem]{Corollary}
 \newtheorem{Proposition}[Theorem]{Proposition}

 \newtheorem{Definition}[Theorem]{Definition}
 \newtheorem*{Definition*}{Definition}

 \newtheorem*{Conjecture*}{Conjecture}

\let\epsilon\varepsilon
\let\kappa=\varkappa
\def\qed{\ifhmode\textqed\fi
      \ifmmode\ifinner\quad\qedsymbol\else\dispqed\fi\fi}
\def\textqed{\unskip\nobreak\penalty50
       \hskip2em\hbox{}\nobreak\hfil\qedsymbol
       \parfillskip=0pt \finalhyphendemerits=0}
\def\dispqed{\rlap{\qquad\qedsymbol}}

\opn\dis{dis}
\def\pnt{{\raise0.5mm\hbox{\large\bf.}}}

\opn\Lex{Lex}

\begin{document}

 \title{The edge rings of compact graphs }

 \author {Zexin Wang and Dancheng Lu }

 \begin{abstract}  We define a simple graph as compact if it lacks even cycles and satisfies the odd-cycle condition. Our focus is on classifying all compact graphs and examining the characteristics of their edge rings. Let $G$ be a compact graph and $\mathbb{K}[G]$ be its edge ring. Specifically, we demonstrate that the Cohen-Macaulay type and the projective dimension of $\mathbb{K}[G]$ are both equal to the number of induced cycles of $G$ minus one and that the regularity of $\mathbb{K}[G]$ is equal to the matching number of $G_0$. Here, $G_0$ is a graph obtained from $G$ by removing the vertices of degree one successively, such that every vertex in $G_0$ has a degree greater than 1.
 \end{abstract}


\address{Zexin Wang,  School of Mathematics Science, Soochow University, P.R.China} \email{zexinwang6@outlook.com}

\address{Dancheng Lu, School of Mathematics Science, Soochow University, P.R.China} \email{ludancheng@suda.edu.cn}
 \maketitle
\section*{Introduction}
Recently, many authors have investigated the algebraic properties of edge rings of simple graphs.  Consider a simple graph $G=(V,E)$ with vertex set $V=\{x_1,\ldots,x_n\}$ and edge set $E=\{e_1,\ldots,e_r\}$. The {\it edge ring} $\mathbb{K}[G]$ is defined to be the toric ring $\mathbb{K}[x_e\:\; e\in E(G)]\subset \mathbb{K}[x_1,\ldots,x_n]$, where $x_e=\prod_{x_i\in e}x_i$ for all $e\in E(G)$. Let $\mathbb{K}[E(G)]$ (or $\mathbb{K}[E]$ for short) denote the polynomial ring $\mathbb{K}[e_1,\ldots,e_r]$ in variables $e_1,\ldots,e_r$. Then, there is exactly one ring homomorphism  $\phi:  \mathbb{K}[E(G)]\rightarrow \mathbb{K}[V] \quad {\mbox{ such that }} ~e_i\mapsto x_{e_i} ~{i=1,\ldots,r}.$
The kernel of the homomorphism map $\phi$ is called the {\it toric ideal} or the {\it defining ideal} of $\mathbb{K}[G]$ or $G$, which is denoted by $I_G$. It follows  that $\mathbb{K}[G]\cong \mathbb{K}[E(G)]/I_G$. The main focus of these studies is to establish connections between the combinatorial properties of simple graphs and the algebraic properties of their edge rings, see e.g. \cite{BOV,BV,FKT,G,GV,HBO,NN,OH} for some developments in this area.

 In 1999, Ohsugi and Hibi demonstrated in \cite{OH} that  $\mathbb{K}[G]$ is a normal domain if and only if $G$ satisfies the odd-cycle condition.
Recall a simple graph is said to satisfy the {\it odd-cycle} condition if, for every pair of cycles $C_1$ and $C_2$, either $C_1$ and $C_2$ have at least one vertex in common or there is an edge that connects a vertex of $C_1$ to a vertex of $C_2$.  We call a simple graph to be {\it compact} if it not only satisfies the odd-cycle condition but also contains no even cycles.  In this paper, we devote to investigating the properties of the edge rings of compact graphs.

Let $G$ be a compact graph. The main results of this paper can be summarized as follows. Firstly, we demonstrate that the projective dimension and Cohen-Macaulay type of $\mathbb{K}[G]$ are both equal to the number of the induced cycles of $G$ minus one. Additionally, we show that the regularity of $\mathbb{K}[G]$ coincides with the matching number of $G_0$. Here, $G_0$ refers to the graph derived from $G$  by successively removing all vertices of degree one. This finding serves as an interesting complement to the result presented in \cite[Theorem 1 (a)]{HH1}, which states if $G$ is a non-bipartite graph satisfying the odd-cycle condition, the regularity of $\mathbb{K}[G]$ does not exceed the matching number of $G$. Finally, we determine the top graded Betti numbers of $\mathbb{K}[G]$. Here, for a simple graph $G$, a {\em matching} of $G$ is a subset $M \subset E(G)$ where $e \cap e' = \emptyset$ for any distinct edges  $e, e'\in M$, and the {\em matching number} of $G$, denoted by $\mathrm{mat}(G)$, is the maximal cardinality of matchings of $G$.

 The paper is organized as follows. Let $G$ be a compact graph. Section 1 provides a brief overview of toric ideals of graphs and canonical modules.  Section 2 classifies the compact graphs up to the (essentially) same edge rings. In Section 3 we compute the universal Gr$\ddot{\mathrm{o}}$bner bases for the toric ideals of compact graphs and then obtain their initial ideals with respect to some suitable monomial order.  In Section 4, we show that all initial ideals obtained in Section 3 possess a ``good" E-K splitting, enabling us  to present a simple formula for the total Betti numbers of such ideals.  Consequently, the regularity, projection dimension, and an upper bound for the Cohen-Macaulay type of $\mathbb{K}[G]$ are derived. Section 5  provides the top graded Betti numbers for $\mathbb{K}[G]$  by computing the minimal generators of its canonical module. In Section 6,  a question regarding the Betti numbers for $\mathbb{K}[G]$ is posed.

\section{Preliminaries}
In this section, we keep the notions given in Introduction and provide a brief review of the notation and fundamental facts that will be utilized later on.
\subsection{Betti numbers and Canonical modules}
Let $R:=\mathbb{K}[x_1,\ldots,x_n]$ be the polynomial ring in variables $x_1,\ldots,x_n$, which is standard graded. For a finitely generated graded  $R-$module $M$, there exists the minimal graded free resolution of $M$ that has the form:
\begin{equation}\tag{\S}\label{free}0\rightarrow\underset{j\in\mathbb{Z}}{\bigoplus}R[-j]^{\beta_{p,j}(M)}\rightarrow \cdots \rightarrow \underset{j\in\mathbb{Z}}{\bigoplus}R[-j]^{\beta_{1,j}(M)}\rightarrow\underset{j\in\mathbb{Z}}{\bigoplus}R[-j]^{\beta_{0,j}(M)}\rightarrow M \rightarrow 0.\end{equation}
Here, $R[-j]$ is the cyclic free $R$-module generated in degree $j$. The number $\beta_{i,j}(M):={\rm{dim}}_{\mathbb{K}}\mathrm{Tor}_i^R(M,\mathbb{K})_j$ is called the $(i,j)$-th {\it graded Betti number} of $M$ and $\beta_{i}(M):=\sum_{j\in \mathbb{Z}}\beta_{i,j}$ is called the $i$-th {\it total Betti number} of $M$.
Many homological invariants of $M$ can be defined in terms of its minimal graded free resolution. The {\em{Castelnuovo-Mumford regularity}} and {\it projective dimension} of $M$ are defined to be
$$\mbox{reg}\,(M):={\mbox{max}}\,\{j-i\mid \beta_{i,\,j}(M)\neq 0\}$$  and $$\mbox{pdim}\,(M):={\mbox{max}}\,\{i\mid \beta_{i,\,j}(M)\neq 0 \mbox{ for some } j \}.$$
Denote $\mbox{pdim}\,(M)$ by $p$. Then, $\beta_p(M)$ and  $\beta_{p,j}(M), j\in \mathbb{Z}$ are referred to as the {\it top total Betti number} and the {\it top graded Betti numbers} of $M$, respectively.

 By applying the functor $\mathrm{Hom}_R(-,R[-n])$ to the sequence (\ref{free}), we obtain the following complex:
\begin{equation*}\begin{split}0\rightarrow \mathrm{Hom}_R(F_0, R[-n])\rightarrow \mathrm{Hom}_R(F_1,R[-n])\rightarrow \cdots \\ \rightarrow \mathrm{Hom}_R(F_p,R[-n])\rightarrow \mathrm{Ext}^{p}_R(M,R[-n])\rightarrow 0.\end{split}\end{equation*} Here, $F_i$ denotes the free module $\underset{j\in\mathbb{Z}}{\bigoplus}R[-j]^{\beta_{i,j}(M)}$. Assume further that $M$ is  Cohen-Macaulay.
Then, it follows from the local duality (see \cite{BH}) that the above complex is exact and so it is a minimal free resolution of $\mathrm{Ext}^{p}_R(M,R[-n])$. The module $\mathrm{Ext}^{p}_R(M,R[-n])$, also denoted by $\omega_{M}$, is called the {\it canonical module} of $M$. Note that
$$\mathrm{Hom}_R(F_i, R[-n])=\bigoplus_{j\in \mathbb{Z}}\mathrm{Hom}_R(R[-j]^{\beta_{i,j}(M)}, R[-n])=\bigoplus_{j\in \mathbb{Z}}R[-n+j]^{\beta_{i,j}(M)}.$$  Based on these discussions, we can derive the following well-known result.
\begin{Lemma} \label{can} Let $M$ be a Cohen-Macaulay graded $R=\mathbb{K}[x_1,\ldots,x_n]$-module, and $\omega_M$ its canonical module. Assume $p=\mathrm{pdim}(M)$.  Then $\beta_{i,j}(\omega_M)=\beta_{p-i,n-j}(M)$ for all $i,j.$
\end{Lemma}
The {\it Cohen-Macaulay type} of  a finitely generated Cohen-Macaulay $R$-module $M$ is defined to be the number
$$\mathrm{type}{(M)} := \beta_{p}(M)=\beta_0(\omega_{M}) ,$$
where $p$ is the projective dimension of $M$.
In the following, we will consider the case when $M=\mathbb{K}[G]$ as a $\mathbb{K}[E(G)]$-module.

\subsection{Toric ideals of graphs}
Let $G$ be a simple graph, i.e., a finite graph without loops and multiple edges, with vertex set $V(G)$ and edge set $E(G)$. A {\em matching} of $G$ is a subset $M \subset E(G)$ for which $e \cap e' = \emptyset$ for $e \neq e'$ belonging to $M$.  The {\em matching number}, denoted by $\mbox{mat}(G)$, is the maximal cardinality of matchings of $G$.   Recall that a walk of $G$ of length $q$ is a subgraph $W$ of $G$ such that $E(W)=\{\{v_0,v_1\},\{v_1,v_2\},\ldots, \{v_{q-1},v_q\}\}$, where $v_0,v_1,\ldots,v_q$ are vertices of $G$. A walk $W$ of $G$ is even if $q$ is even, and it is closed if $v_0=v_q$. A cycle is a special  closed walk  with edge set $\{\{v_0,v_1\},\{v_1,v_2\},\{v_{q-1},v_q=v_0\}\}$ such that $v_1,\ldots,v_q$ are pairwise distinct and $q\geq 3$. A cycle is called {\it even} (resp. {\it odd})   if $q$ is even (resp. odd). For a subset $W$ of $V(G)$, the {\it induced subgraph} $G_W$ is the graph with vertex set $W$ and for every pair $x,y\in W$, they are adjacent in $G_W$ if and only if they are adjacent in $G$. An induced cycle of $G$ is a cycle $C$ where no two non-consecutive vertices of $C$ are adjacent in $G$.

The generators of the toric ideal of $I_G$ are binomials which are tightly related to even closed walks in $G$. Given an even closed walk $W$ of $G$ with
$$E(W)=\{\{v_0,v_1\},\{v_1,v_2\},\ldots,\{v_{2q-2},v_{2q-1}\},\{v_{2q-1},v_0\}\},$$
we associate $W$ with the binomial defined by
$$f_W:=\prod\limits_{j=1}^{q}e_{2j-1}-\prod\limits_{j=1}^{q}e_{2j},$$
where $e_j=\{v_{j-1},v_j\}$ for $1\leq j\leq 2q-1$ and $e_{2q}=\{v_{2q-1},v_0\}$. A binomial $f = u - v \in I_{G}$ is called a {\it primitive binomial} if there is no binomial $g = u' - v' \in I_{G}$ such that $u'|u$ and $v'|v$. An even closed walk $W$ of $G$ is a {\it primitive even closed walk} if its associated binomial $f_{W}$ is a primitive binomial in $I_G$. It is known that the set  $$\{f_W\:\; W \mbox{ is a primitive even closed walks of } G\}$$ is the universal  Gr$\ddot{\mathrm{o}}$bner base of $I_G$ by e.g. \cite[Proposition 10.1.10]{V} or \cite[Proposition~5.19]{EH}.  In particular, it is a Gr$\ddot{\mathrm{o}}$bner base of $I_G$ with respect to any monomial order. The set of primitive even walks of a graph $G$ was described in \cite{HHO} explicitly.

\begin{Lemma}\label{PECW}\cite[Lemma~5.11]{HHO}
A primitive even closed walk $\Gamma$ of $G$ is one of the following:
\begin{enumerate}
\item[$(i)$]  $\Gamma$ is an even cycle of $G$;
\item[$(ii)$] $\Gamma$ = $(C_{1}, C_{2})$, where each of $C_{1}$ and $C_{2}$ is an odd cycle of $G$ having exactly
one common vertex;
\item[$(iii)$] $\Gamma$ = $(C_{1}, \Gamma_{1}, C_{2},\Gamma_{2})$, where each of $C_{1}$ and $C_{2}$ is an odd cycle of $G$ with
$V (C_{1})\cap V (C_{2}) =\emptyset$ and where $\Gamma_{1}$ and $\Gamma_{2}$ are walks of $G$ of the forms $\Gamma_{1} =(e_{i_{1}},\ldots,e_{i_{r}})$ and $\Gamma_{1} =(e_{i^{'}_{1}} ,\ldots,e_{i^{'}_{r^{'}}} )$ such that $\Gamma_{1}$ combines $j\in e_{i_{1}}\cap e_{i^{'}_{r^{'}}}\cap V (C_{1})$ with $j^{'}\in e_{i_{r}}\cap e_{i^{'}_{1}}\cap V (C_{2})$  and $\Gamma_{2}$ combines $j^{'}$ with $j$. Furthermore, none of the vertices belonging to $V(C_{1})\cup V(C_{2})$ appears in each of $e_{i_{1}}\backslash
\{j\}$, $e_{i_{2}}$,$\ldots$,$e_{i_{r-1}}$, $e_{i_{r}}\backslash\{j^{'}\}$, $e_{i^{'}_{1}}\backslash\{j\}$, $e_{i^{'}_{2}}$,\ldots,$e_{i^{'}_{r-1}}$,$e_{i^{'}_{r^{'}}}\backslash\{j^{'}\}$.
\end{enumerate}
\end{Lemma}

We would like to note  that in ($iii$) the sum of lengths of $\Gamma_1$ and $\Gamma_2$ must be even in order  to ensure it is indeed an even closed walk.

\subsection{Edge Cones and Canonical modules}

Let $G$ be a simple graph with vertex set $V(G)=\{1,\ldots,n\}$ and edge set $E(G)$. For any $f =\{i, j\} \in E(G)$ denote $v_{f}= \mathbf{e}_{i}+ \mathbf{e}_{j}$, where $\mathbf{e}_{i}$ is the $i$th unit vector of $\RR^{n}$.
The edge cone of $G$, denoted by $\RR_{+}(G)$, is defined to be the cone of $\RR^{n}$ generated by $\{v_{f}\mid  f \in E(G)\}$. In other words, $$\RR_{+}(G) = \{ \sum\limits_{f \in E(G)}{a_{f}v_{f}} \mid a_{f} \in \RR_{+} \mbox{ for all } f\in E(G)\}.$$
If $G$ satisfies the odd-cycle condition, then the edge ring $\mathbb{K}[G]$ is normal, see \cite{OH}, and particularly,  $\mathbb{K}[G]$ is Cohen-Macaulay, see \cite[Theorem 6.3.5]{BH}. It follows that the ideal of $\mathbb{K}[G]$ generated all the monomials $x^{\alpha}$ with  $\alpha \in \ZZ^n\cap \mbox{relint}(\RR_{+}(G))$ is the canonical module of $\mathbb{K}[G]$,
see e.g. \cite[section 6.3]{BH} for the details.

Let us describe the cone $\RR_{+}(G)$ in terms of linear inequalities.
For the description, we need to introduce some more notions on graphs.
\begin{itemize}
\item For a subset $W \subset V(G)$, let $G \setminus W$ be the subgraph induced on $V(G) \setminus W$.
If $W=\{k\}$, then we write $G \setminus k$ instead of $G \setminus \{k\}$.
\item For $j \in V(G)$, let $N_G(j)=\{i \in V(G)\mid \{i,j\} \in E(G)\}$, and for any subset $W \subset V(G)$, let $N_G(W)=\bigcup\limits_{k \in W}N_G(k)$.
\item A non-empty subset $T \subset V(G)$ is called an \textit{independent set} if $\{j,k\} \not\in E(G)$ for any $j,k \in T$.
\item We call a vertex $j$ of $G$ \textit{regular} if each connected component of $G \setminus j$ contains an odd cycle.
\item We say that an independent set $T$ of $V(G)$ is a \textit{fundamental} set if
\begin{itemize}
\item the bipartite graph on the vertex set $T \cup N_G(T)$ with the edge set $E(G)\cap \{\{j,k\}\mid j \in T, k \in N_G(T)\}$ is connected, and
\item either $T \cup N_G(T)=V(G)$ or each of the connected components of the graph $G \setminus (T \cup N_G(T))$ contains an odd cycle.
\end{itemize}
\end{itemize}

It follows from \cite[Theorem 3.2]{V1} or (\cite[Theorem 1.7 (a)]{OH}) that $\RR_{+}(G)$ consists of the elements $(x_1, \ldots, x_n) \in \RR^n$ satisfying all the following inequalities:
\begin{equation}\label{eq:ineq}\tag{$\Delta$}
\begin{split}
x_u &\geq 0 \;\;\text{ for any regular vertex }u; \\
\sum\limits_{v \in N_G(T)}x_v &\geq \sum\limits_{u \in T}x_u \;\;\text{ for any fundamental set }T.
\end{split}
\end{equation}

\subsection{E-K splitting} Based on the approach in \cite{EK}, Eliahou and Kervaire introduced the notion
of splitting a monomial ideal.
\begin{Definition}\label{DEKSP}
\em Let $I, J$ and $K$ be monomial ideals such that
$G(I)$, the unique set of minimal generators of $I$, is the disjoint union of $G(J)$ and $G(K)$. Then $I=J+K$ is an \textbf{Eliahou-Kervaire splitting} (abbreviated as ``E-K splitting'') if
there exists a splitting function \[ G(J \cap K) \to G(J) \times G(K) \] sending
$w \mapsto (\phi(w),\psi(w))$ such that

\begin{enumerate}
\item $w=\mathrm{lcm}(\phi(w),\psi(w))$ for all $w \in G(J \cap K)$, and
\item for every subset $\emptyset \neq S \subset G(J \cap K)$, $\mathrm{lcm}(\phi(S))$ and $\mathrm{lcm}(\psi(S))$
strictly divide $\mathrm{lcm}(S)$.
\end{enumerate}
\end{Definition}

\begin{Lemma} \cite[Proposition 3.1]{EK}
Let $I=J+K$ be an E-K splitting. Then, for all $i\geq0$,
\begin{equation*} \label{splitformula}
(*)~~\hspace{.5cm}
\beta_{i}(I) = \beta_{i}(J)+\beta_{i}(K)+\beta_{i-1}(J \cap K), \\
\beta_{i,j}(I) = \beta_{i,j}(J)+\beta_{i,j}(K)+\beta_{i-1,j}(J \cap K),
\end{equation*}
where $\beta_{-1,j}(J\cap K)=0$ for all $j$ by convention.
\end{Lemma}

\section{A classification of compact graphs}

In this section, we aim to classify all the compact graphs up to the essentially same edge rings.
We start by presenting the following straightforward observation, which we will not provide a proof for.

\begin{Lemma}\label{remove} Let $x_1$ be a vertex of degree one in a simple graph $G$ and let $G'$ be the graph obtained from $G$ by removing $x_1$. Then $I_G$ and $I_{G'}$ have the same set of minimal binomial generators. More precisely, if $G$ has edge set $\{e_1,\ldots,e_r\}$ with $x_1\in e_r$, then $I_G=I_{G'}\cdot\mathbb{K}[e_1,\ldots,e_r]$ and $\mathbb{K}[G]\cong \mathbb{K}[G']\otimes_{\mathbb{K}}\mathbb{K}[e_r]$. Here,   both $\mathbb{K}[e_1,\ldots,e_r]$ and $\mathbb{K}[e_r]$ are polynomial rings by definitions.
\end{Lemma}
This observation indicates that the removal of vertices with a degree of one does not essentially alter the edge ring.  Given a simple graph $G$, by iteratively removing all vertices of degree one, we obtain a new graph, denoted as $G_0$,  where every remaining vertex has a degree greater than one. It is evident that $G$ and $G_0$ essentially share the same edge ring by Lemma~\ref{remove}.  From this point forward, we will solely focus on simple graphs in which every vertex has a degree greater than one.

\begin{Definition}\em Let $G$ be a connected simple graph where every vertex has a degree greater than one. We call $G$ to be a {\it compact} graph if it does not contain any even cycles and satisfies the odd-cycle condition.
\end{Definition}

To give a complete classification of compact graphs we need a series of lemmas, in which the graph $G$ is always a compact graph.

 \begin{Lemma} \label{2.3} Every cycle of $G$ is an induced cycle. \end{Lemma}
\begin{proof}
Let $C$ be a cycle of $G$ that is not induced. We may label the vertices of $C$ as $v_1-v_2-\cdots-v_{2s+1}-v_1$ in such a way there exists an index $i \notin \{2, 2s+1\}$ for which $v_1$ is adjacent to $v_i$. Consequently, we obtain two distinct cycles: $v_1-v_2-\cdots-v_{i}-v_1$ and $v_{i}-v_{i+1}-\cdots-v_{2s+1}-v_1-v_{i}$. These cycles have lengths $i$ and $2s+3-i$ respectively, and by virtue of this, precisely one of them must be an even cycle.  This is impossible by our assumption.
\end{proof}

 \begin{Lemma} \label{2.4} For any distinct cycles $C_1$ and $C_2$ of $G$,  one has $|V(C_1)\cap V(C_2)|\leq 1$.\end{Lemma}
\begin{proof}
On the contrary, let us assume the existence of two distinct cycles, $C_1: v_1-v_2-\cdots-v_{2s+1}-v_1$ and $C_2: u_1-u_2-\cdots-u_{2t+1}-u_1$, such that the intersection of their vertex sets has at least two elements. Without loss of generality, we may assume $u_1 = v_1$.

Now, define $i_1$ as the smallest integer $i \geq 2$ such that $v_i$ belongs to the vertex set of $C_2$. We assert that $i_1 = 2$. To see this, suppose $i_1 \geq 3$. Then, since $v_{i_1} = u_{j_1}$ for some $j_1\neq 1$, we can construct two cycles: $u_1 = v_1-v_2-\cdots-v_{i_1} = u_{j_1}-u_{j_1+1}-\cdots-u_{2t+1}-u_1$ and $u_1 = v_1-v_2-\cdots-v_{i_1} = u_{j_1}-u_{j_1-1}-\cdots-u_2-u_1$. However, this is impossible because exactly one of these cycles is even, contradicting our assumption. Therefore, $i_1 = 2$, which means $v_2$ is a vertex of $C_2$. Since $C_2$ is an induced cycle, $v_2$ must be either $u_2$ or $u_{2t+1}$. By relabeling the vertices of $C_2$ if necessary, we can assume $v_2 = u_2$.

If  $|V(C_1)\cap V(C_2)|=2$, i.e., $V(C_1)\cap V(C_2)=\{v_1,v_2\}$, then an even cycle can be constructed:
\[
u_1 = v_1 - v_{2s+1} - v_{2s} - \cdots - v_2 = u_2 - u_3 - \cdots - u_{2t+1} - u_1 = v_1.
\]
This contradicts our assumption, indicating that the  $|V(C_1)\cap V(C_2)|\geq 3$.

Next, let $i_2$ be the smallest integer $i\geq 3$ such that $v_i$ belongs to the vertex set of $C_2$. We assert that $i_2=3$. If $i_2$ is greater than 3, then two cycles are formed:
\[
u_2 = v_2 - v_3 - \cdots - v_{i_2} = u_{j_2} - u_{j_2+1} - \cdots - u_{2t+1} - u_2
\]
and
\[
u_2 = v_2 - v_3 - \cdots - v_{i_2} = u_{j_2} - u_{j_2-1} - \cdots - u_3 - u_2,
\]
where $j_2$ satisfies $v_{i_2} = u_{j_2}$. However, one of these cycles must be even, creating a contradiction.

Therefore, $i_2$ is indeed 3, implying that $v_3$ is a vertex of $C_2$. Since $C_2$ is an induced cycle, $v_3$ must be $u_3$.

Continuing this process, we ultimately arrive at the conclusion that $C_1$ and $C_2$ are identical, which is a contradiction.
  \end{proof}

\begin{Lemma}\label{2.5} Every vertex of $G$ belongs to at least one cycle.\end{Lemma}

 \begin{proof}

Assume on the contrary that there exists a vertex $v$ in $G$ that does not belong to any cycle. We first establish the claim that for any $u \in N_G(v)$, there exists a path $v-u-u_1-\cdots-u_s$ such that  $u_s$ belongs to at least one cycle in $G$, whereas each one of vertices $u,u_1,\ldots,u_{s-1}$ does not belong to any cycle in $G$.

Let $u \in N_G(v)$. If $u$  belongs to at least one cycle in $G$, we are done. (In this case, $s=0$). If  $u$ does not belong to any cycle in $G$, then, since $\deg(u) \geq 2$, there exists a vertex $u_1$ such that $u_1 \in N_G(u) \setminus \{v\}$. If $u_1$ belongs to a cycle, we are done. Otherwise, if $u_1$ does not belong to any cycle, we select $u_2 \in N_G(u_1) \setminus \{u\}$. If $u_2 = v$, then $v$ belongs to the cycle $v-u-u_1-u_2=v$, contradicting our initial assumption.

Therefore, $v-u-u_1-u_2$ is a path. If $u_2$ belongs to a cycle, we are done. If not, we continue this process, selecting $u_3 \in N_G(u_2) \setminus \{u_1\}$ and so on. Since $G$ is a finite graph, there exists an integer $s\geq 0$ such that $v-u-u_1-\cdots-u_s$ is a path satisfying the condition required, thereby proving our claim.

Now, since $\deg(v) \geq 2$, we can select distinct vertices $u$ and $w$ in $N_G(v)$. By applying the previous claim, there exist paths $v-u-u_1-\cdots-u_s$ and $v-w-w_1-\cdots-w_t$, where each vertex in  $\{v,u,w, u_1,\ldots,u_{s-1},w_1,\ldots,w_{t-1}\}$ does not belong to any cycle, but $u_s$ and $w_t$ belong to cycles $C_1$ and $C_2$ respectively. Clearly, $C_1$ and $C_2$ must be disjoint, as otherwise $v$ would belong to a cycle, contradicting our assumption.

However, this implies that there is an edge connecting $C_1$ and $C_2$, due to the odd-cycle condition in $G$. This is again a contradiction, and hence we conclude that every vertex in $G$ must belong to at least one cycle.
\end{proof}

 \begin{Lemma}\label{2.6}{\it If $v$ belong to the vertex set of  a cycle $C$ with $\deg(v)\geq 3$, then for any $u\in N_G(v)\setminus V(C)$ and for any cycle $C_1$ passing through $u$, one has either $V(C_1)\cap V(C)=\emptyset$ or  $V(C_1)\cap V(C)=\{v\}.$ } \end{Lemma}

\begin{proof}
If the intersection of the vertex sets $V(C_1)$ and $V(C)$ is neither empty nor equal to $\{v\}$, then by Lemma \ref{2.4}, it must be $\{w\}$ for some vertex $w$ distinct from $u$. Since $u$ and $w$ both belong to  $V(C_1)$, there exists a path $u-u_1-\cdots-u_s-w$, with  each vertex along this path belonging to $V(C_1)$. Similarly, as  $v$ and $w$ both belong to  $V(C)$, there exists a path $w-w_1-\cdots-w_t-v$, where each vertex in this path is an element of $V(C).$ Consequently, we can construct a cycle $u-u_1-\cdots-u_s-w-w_1-\cdots-w_t-v-u$ that necessarily intersects $C$ at at least two vertices: $v$ and $w$.  This is impossible by Lemma \ref{2.4}.
\end{proof}

We now study what happens if a cycle contain two vertices of degree at least 3. In the  proof of the following lemma, we use notation $u--_C--v$ to represent any  of two paths from $u$ to $v$ within the  cycle $C$ unless  otherwise specified.
  \begin{Lemma}\label{5} Let $v_1$ and $v_2$ be distinct vertices of degree at least 3 that belong to the vertex set of a  cycle $C$. Then $v_1$ is adjacent to $v_2$. Furthermore, for each $i=1,2$,  if  $C_i$ is  a cycle that passes through a vertex, say $u_i$,  in $N_G(v_i)\setminus V(C)$, then  $C_i$ must passes through $v_i$. Moreover, $V(C_1) \cap V(C_2) = \emptyset$.
  \end{Lemma}

 \begin{proof}

 It is evident that $u_1 \neq u_2$, for otherwise, a cycle $v_1 --_C-- v_2 - u_2 = u_1- v_1$ would be formed.  This cycle would intersect $C$ at least at two vertices, $v_1$ and $v_2$, which is impossible according to Lemma~\ref{2.4}.

We next show $V(C_1) \cap V(C_2) = \emptyset$. If not, we may assume $V(C_1) \cap V(C_2)=\{x\}$.
There exists a least one path from $x$ to $u_1$ within $C_1$ that does does not pass through $v_1$, and we denote it by $x--_{C_1}--u_1$.  Additionally, let $u_2--_{C_2}--x$ denote a path within $C_2$ that does not pass through $v_2.$ Then,  since $V(C)\cap V(C_1)\subseteq \{v_1\}$, every vertex in the  path $x--_{C_1}--u_1$  does not belong to $V(C)$. Similarly, every vertex in the  path  $u_2--_{C_2}--x$  does not belong to $V(C)$.   If $x\notin \{u_1,u_2\}$,  then we obtain the  cycle: $x--_{C_1}--u_1-v_1--_C--v_2-u_2---_{C_2}-x.$  It's impossible because this cycle intersects with $C$ at least at two vertices, $v_1$ and $v_2$.   Either the case  $x=u_1$ or the $x=u_2$ leads to a similar contradiction. Hence, $V(C_1) \cap V(C_2) = \emptyset$.

Next, we demonstrate that $C_1$ necessarily passes through $v_1$. If not, then by Lemma \ref{2.6}, $V(C_1) \cap V(C) = \emptyset$. Since $V(C_1) \cap V(C_2) = \emptyset$, there exists an edge $\{x_1, x_2\}$ such that $x_i \in V(C_i)$ for $i = 1, 2$. Assuming $V(C_2) \cap V(C) \neq \emptyset$, it implies $V(C_2) \cap V(C) = \{v_2\}$.

Considering various possibilities, we have:

1. If $x_1 \neq u_1$ and $x_2 \neq v_2$, we arrive at the cycle: $x_1--_{C_1}-- u_1 - v_1 --_{C}-- v_2 --_{C_2}-- x_2 - x_1$.

2. If $x_1 = u_1$ and $x_2 \neq v_2$, the cycle is: $x_1 = u_1 - v_1 --_{C}-- v_2 --_{C_2}-- x_2 - x_1$.

3. If $x_1 = u_1$ and $x_2 = v_2$, the cycle becomes: $x_1 = u_1 - v_1 --_{C}-- v_2 = x_2- x_1$.

4. If $x_1 \neq u_1$ and $x_2 = v_2$, the cycle is: $x_1 --_{C_1}-- u_1- v_1--_{C} -- v_2 = x_2-- x_1$.

However, all these scenarios are impossible due to Lemma~\ref{2.4}. This proves $V(C_2) \cap V(C) = \emptyset$. However, we can prove this is also impossible analogously. This establishes that $C_1$ must pass through $v_1$. Similarly, we can also demonstrate that $C_2$ necessarily passes through $v_2$.

Finally, we show that $v_1$ is adjacent to $v_2$.  Suppose that $v_1$ is not adjacent to $v_2$. Since $V(C_1) \cap V(C_2) = \emptyset$, there is an edge $\{x_1,x_2\}$ such that $x_i\in V(C_i)$ for $i=1,2$. Note that  either $x_1\neq v_1$ or $x_2\neq v_2$. We may assume $x_1\neq v_1$. Then we obtain a cycle: either $x_1--_{C_1}--v_1--_{C}--v_2--_{C_1}--x_2-x_1$ if $x_2\neq v_2$, or $x_1--_{C_1}--v_1---v_2=x_2-x_1$ if $v_2=x_2$. This is again a contradiction.
\end{proof}

  A cycle is isolated if each of its vertices has a degree of $2$. For convenience, we say  a cycle of $G$ to be {\it almost-isolated}  if it has  exactly one vertex of degree $\geq 3$.

 \begin{Lemma} \label{2.8}If  $v_1$ is a vertex with $\deg(v_1)\geq 3$, then $v_1$  belongs to at least one almost-isolated cycle $C$.
 \end{Lemma}

  \begin{proof}
By Lemma \ref{2.5} we may let $C$ be a cycle passing $v_1$. If $C$ is almost-isolated, our proof is complete. If $C$ is not almost-isolated, we let $v_2$ be a vertex in $V(C)$ other than $v_1$ with degree at least 3.

For each $i=1,2$, we let $u_i\in N_G(v_i)\setminus V(C)$ and let $C_i$ be a cycle passing through $u_i$. Then, according to Lemma~\ref{5}, we have $C_i$ passes $v_i$ for $i=1,2$ and $V(C_1)\cap V(C_2)=\emptyset$. Moreover, $v_1$ is adjacent to $v_2$.

It remains to show that $C_1$ is a almost-isolated cycle. Suppose on the contrary that there is a vertex $v_1\neq v_3\in V(C_1)$ with a degree at least $3$. Then $v_1$ and $v_3$ are vertices of degree at least $3$ that belong to $C_1$. Note that $C_2$ passes through the vertex $v_2$, which belongs to $N_G(v_1)\setminus V(C_1)$. It follows that $C_2$ passes through $v_1$ by Lemma~\ref{5}. However, this is contradicted to the fact that $V(C_1)\cap V(C_2)=\emptyset$, completing the proof.
\end{proof}

\begin{Lemma}  \label{2.9} If $v_1$ and $v_2$ are distinct vertices with a degree of at least 3, then $v_1$ is adjacent to $v_2$.\end{Lemma}
\begin{proof} By Lemma~\ref{2.8}, there exist  almost-isolated cycles $C_1$ and $C_2$ such that $v_i$ belongs to $C_i$ for $i = 1, 2$. It is clear that $V(C_1)\cap V(C_2)=\emptyset$. According to the odd-cycle condition, there exists an edge $\{u, w\}$ connecting a vertex $u$ from $V(C_1)$ to a vertex $w$ from $V(C_2)$. This implies both $u$ and $w$ have a degree of at least $3$.  Since all vertices of $C_1$ apart from $v_1$ have a degree of 2, it follows that $u$ must be $v_1$. By a similar argument, we conclude that $w$ is $v_2$. Hence, $v_1$ is adjacent to $v_2$.\end{proof}

\begin{Lemma} \label{2.10}  There are no more than three vertices with a degree of at least 3.\end{Lemma}
\begin{proof}
Let $v_1, \ldots, v_k$ be all the vertices with a degree of at least 3. By  Lemma \ref{2.9}, the induced subgraph on $\{v_1, \ldots, v_k\}$ forms a complete graph. However, if $k$ were to exceed 3 (i.e., $k \geq 4$), this induced subgraph would contain a 4-cycle, which is a contradiction.
\end{proof}

We are now ready to present a complete classification for compact graphs.

\begin{Definition}\em
   A vertex in a compact graph is referred as a {\it big} vertex if its degree exceeds 2. A compact graph is classified as type $i$ if it possesses exactly $i$ big vertices, where $i$ ranges from 0 to 3.

\end{Definition}

\begin{Theorem}\label{main1}
Every compact graph falls into type $i$ for some $i\in \{0,1,2,3\}$.  Furthermore, the four distinct categories of compact graphs can be characterized as follows:

\begin{enumerate}
\item A compact graph of type 0 is simply an odd cycle.

\item A compact graph of type 1 is a finite collection of odd cycles that share a common vertex.

\item A compact graph of type 2 consists of two disjoint compact graphs of type 1, where the two big vertices are connected either by an edge or by an edge as well as a path of even length.

  \item   A compact graph of type 3 consists of three disjoint compact graphs of type 1, where every pair of big vertices is connected by an edge.

\end{enumerate}
\end{Theorem}

\begin{proof}  The first statement follows from Lemma~\ref{2.10}.

The proofs of (1) and (2) are straightforward.

(3) Let $G$ be a compact graph of type 2 and let $u$ and $v$ denote its big vertices.  Then $u$ is adjacent to $v$ by Lemma~\ref{2.9}. Consider all the almost-isolated cycles passing through $u$ and   all the almost-isolated cycles passing through $v$. Let $V_1$ denote the set of  all the vertices of these cycles. If $V_1=V(G)$, then $G$ consists of two disjoint compact graphs of type 1 where the  two big vertices forms an edge.

 Suppose that $V_1 \neq V(G)$. Then we define $V_2 = V(G) \setminus V_1$.  Fix an arbitrary vertex $x \in V_2$. By Lemma~\ref{2.5}, there exists a cycle $C$ in $G$ such that $x$ belongs to the vertex set of $C$. Since $C$ is neither isolated nor almost-isolated, it is straightforward to deduce that $u$ and $v$ must also belong to the vertex set of $C$. By Lemma~\ref{2.3}, $C$ is an induced cycle, meaning that no two non-consecutive vertices of $C$ are adjacent in $G$. Hence, $C$ may be written as $u-x_1-\cdots-x_{2k-1}-v-u$, where $k \geq 1$ and $x_1, \ldots, x_{2k-1}$ are distinct vertices in $V_2$.

Now, consider any other vertex $y \in V_2$. Then $y$ also belongs to the vertex set of a cycle that contains both $u$ and $v$. By Lemma~\ref{2.4}, this cycle  must be $C$. Consequently, $y\in \{x_1,\ldots,x_{2k-1}\}$  and it follows that $V_2=\{x_1, \ldots, x_{2k-1}\}$. Therefore, $G$ consists of two disjoint compact graphs of type 1, where the two big vertices $u$ and $v$ are connected not only by an edge but also by a path given by $u-x_1-\cdots-x_{2k-1}-v$.

(3) Let $G$ be a compact graph of type 3, with $u,v$ and $w$ denoting its big vertices. It suffices to show that the set of vertices of almost-isolated cycles must be $V(G)$. Suppose on the contrary that there exists  a vertex $x$ that does not belong to the vertex set of any almost-isolated cycle of $G$. Then $x$ must be a vertex of a cycle that contains at least two big vertices. However, the only such cycle in $G$ is  $u-v-w-u$. It follows that $x$ must be one of $u,v,w$, a contradiction.
  \end{proof}

We introduce some more notation for the later use.
Suppose $\underline{p}=(p_1,\ldots,p_m)$, $\underline{q}=(q_1,\ldots,q_n)$  and $\underline{r}=(r_1,\ldots,r_k)$ are positive integral vectors with dimensions $m,n$ and $k$ respectively.
We denote a compact graph of type 1, where the odd cycles have lengths $2p_1+1,\ldots 2p_m+1$ respectively, as $A_{\underline{p}}$ or $A_{p_1,\ldots,p_m}$.
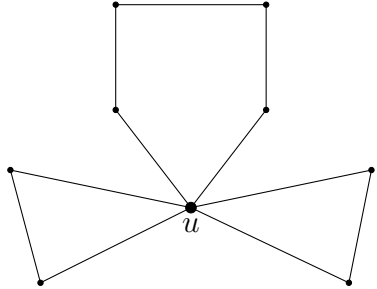
\begin{figure}[ht]
\centering
\begin{tikzpicture}
\draw[black, thin] (3,1) -- (5,2) -- (2.6,2.5)-- cycle;
\draw[black, thin] (5,2) -- (4,3.3) -- (4,4.7) -- (6,4.7)--(6,3.3)-- cycle;
\draw[black, thin] (5,2) -- (7.4,2.5) -- (7.1,1)-- cycle;
\filldraw [black] (3,1) circle (1pt);
\filldraw [black] (5,2) circle (2pt);
\filldraw [black] (2.6,2.5) circle (1pt);
\filldraw [black] (4,3.3) circle (1pt);
\filldraw [black] (6,3.3) circle (1pt);
\filldraw [black] (4,4.7) circle (1pt);
\filldraw [black] (6,4.7) circle (1pt);
\filldraw [black] (7.4,2.5) circle (1pt);
\filldraw [black] (7.1,1) circle (1pt);
\draw (5,2) node[anchor=north]{$u$};
\end{tikzpicture}
\caption{The graph $A_{(1,2,1)}$}
\label{fig1}
\end{figure}

By $B_{\underline{p}:\underline{q}}^0$  we mean  a
compact graph of type 2 where the two disjoint compact graphs of type 1 that compose it are $A_{\underline{p}}$ and $A_{\underline{q}}$ and where the two big vertices are connected by an edge. Furthermore, if $s>0$ is an even number, then  $B_{\underline{p};\underline{q}}^s$ represents the graph obtained by appending a path of length $s$ connecting two big vertices to $B_{\underline{p}:\underline{q}}^0$.

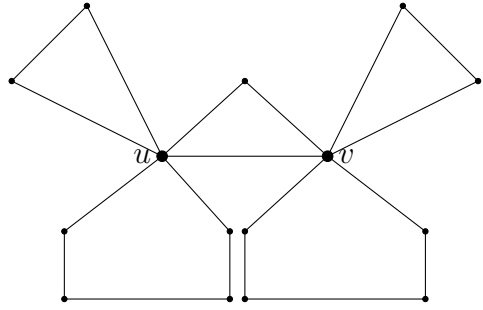
\begin{figure}[ht]
\centering
\begin{tikzpicture}
\draw[black, thin] (3.7,1) -- (5,2) -- (5.9,1)-- (5.9,0.1)--(3.7,0.1)-- cycle;
\draw[black, thin] (5,2) -- (3,3) -- (4,4)-- cycle;
\draw[black, thin] (5,2) -- (6.1,3) -- (7.2,2)-- cycle;
\draw[black, thin] (7.2,2) -- (9.2,3) -- (8.2,4)-- cycle;
\draw[black, thin] (8.5,1) -- (7.2,2) -- (6.1,1)-- (6.1,0.1)--(8.5,0.1)-- cycle;
\filldraw [black] (5,2) circle (2pt);
\filldraw [black] (7.2,2) circle (2pt);
\filldraw [black] (9.2,3) circle (1pt);
\filldraw [black] (8.2,4) circle (1pt);
\filldraw [black] (8.5,1) circle (1pt);
\filldraw [black] (6.1,1) circle (1pt);
\filldraw [black] (6.1,0.1) circle (1pt);
\filldraw [black] (8.5,0.1) circle (1pt);
\filldraw [black] (3.7,1) circle (1pt);
\filldraw [black] (5.9,1) circle (1pt);
\filldraw [black] (5.9,0.1) circle (1pt);
\filldraw [black] (3.7,0.1) circle (1pt);
\filldraw [black] (3,3) circle (1pt);
\filldraw [black] (4,4) circle (1pt);
\filldraw [black] (6.1,3) circle (1pt);
\draw (5,2) node[anchor=east]{$u$};
\draw (7.2,2) node[anchor=west]{$v$};
\end{tikzpicture}
\caption{The graph $B^{2}_{(2,1):(2,1)}$}
\label{fig2}
\end{figure}

 A compact graph of type 3  is denoted by $C_{\underline{p}:\underline{q}:\underline{r}}$ if the three disjoint compact graphs of type 1 that make up it are $A_{\underline{p}}$, $A_{\underline{q}}$ and $A_{\underline{r}}$ respectively.

\begin{figure}[ht]
\centering
\begin{tikzpicture}
\draw[black, thin] (3.7,1) -- (5,2) -- (5.8,1)-- (5.8,0.1)--(3.7,0.1)-- cycle;
\draw[black, thin] (5,2) -- (3,2.2) -- (3.3,3.2)-- cycle;
\draw[black, thin] (5,2) -- (6.1,4) -- (7.2,2)-- cycle;
\draw[black, thin] (7.2,2) -- (9.2,2.2) -- (8.9,3.2)-- cycle;
\draw[black, thin] (8.5,1) -- (7.2,2) -- (6.2,1)-- (6.2,0.1)--(8.5,0.1)-- cycle;
\draw[black, thin] (6.1,4) -- (4,5.5) -- (3.5,4.5)-- cycle;
\draw[black, thin] (6.1,4) -- (8.2,5.5) -- (8.7,4.5)-- cycle;
\filldraw [black] (5,2) circle (2pt);
\filldraw [black] (7.2,2) circle (2pt);
\filldraw [black] (6.1,4) circle (2pt);
\filldraw [black] (8.2,5.5) circle (1pt);
\filldraw [black] (8.7,4.5) circle (1pt);
\filldraw [black] (6.2,0.1) circle (1pt);
\filldraw [black] (8.5,0.1) circle (1pt);
\filldraw [black] (8.5,1) circle (1pt);
\filldraw [black] (4,5.5) circle (1pt);
\filldraw [black] (3.5,4.5) circle (1pt);
\filldraw [black] (3.7,1) circle (1pt);
\filldraw [black] (5.8,1) circle (1pt);
\filldraw [black] (5.8,0.1) circle (1pt);
\filldraw [black] (3.7,0.1) circle (1pt);
\filldraw [black] (3,2.2) circle (1pt);
\filldraw [black] (3.3,3.2) circle (1pt);
\filldraw [black] (9.2,2.2) circle (1pt);
\filldraw [black] (8.9,3.2) circle (1pt);
\filldraw [black] (8.5,1) circle (1pt);
\filldraw [black] (7.2,2) circle (1pt);
\filldraw [black] (6.2,1) circle (1pt);
\draw (5,2) node[anchor=south]{$u$};
\draw (7.2,2) node[anchor=south]{$v$};
\draw (6.1,4) node[anchor=south]{$w$};
\end{tikzpicture}
\caption{The graph $C_{(2,1):(1,1):(2,1)}$}
\label{fig3}

\end{figure}
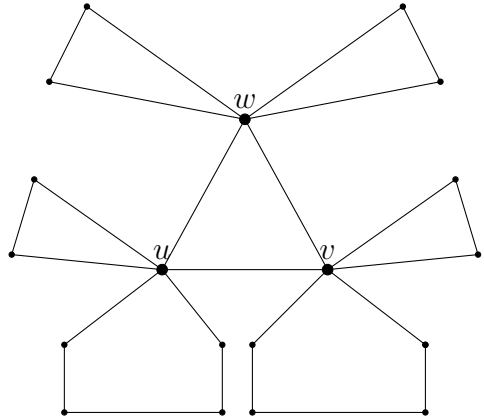

\section{Universal gr\"{o}bner bases and initial ideals}

In this section, we will discuss the universal Gr\"{o}bner  bases and initial ideals of toric ideals of compact graphs, with respect to a specific monomial order. The main objective of this section is to identify suitable monomial orders that yield initial ideals with a favorable E-K splitting, as demonstrated in the subsequent section.

 \subsection{Compact graphs of type 1} \label{3.1} Given positive integers $m \geq 2$ and $p_{1}, \ldots, p_{m} $, we use $A$ to denote
the graph $A_{p_{1}, \ldots, p_{m}}$ for short. Thus $A$ has vertex set
\[
V(A)=\{u\} \cup \{u_{i,j}\mid 1 \leq i \leq m, 1 \leq j \leq 2p_{i}\}
\]
and edge set
\begin{align*}
E(A)=\{& \{u_{i,j},u_{i,j+1}\} \mid 1 \leq i \leq m, 1 \leq j \leq 2p_{i}-1 \}\\
&\cup \{ \{u,u_{i,1}\}, \{u,u_{i,2p_{i}}\} \mid  1 \leq i \leq m \}.
\end{align*}
We label the edges of $A$ as follows. For $i\in{\{1,\ldots,m\}}$, we let $e_{i,1}=\{u,u_{i,1}\}$ and $e_{i,2p_{i}+1}=\{u,u_{i,2p_{i}}\}$.  For $i\in{\{1,\ldots,m\}}$ and $j\in{\{1,\ldots,2p_{i}-1\}}$ let $e_{i,j+1}=\{u_{i,j},u_{i,j+1}\}$. For $1\leq i,j\leq m$, we put $$e_i'=e_{i,1}e_{i,3}\cdots e_{i,2p_i+1} \mbox{ and } e_j''=e_{j,2}e_{j,4}\cdots e_{j,2p_j}.$$

 \begin{Lemma}\label{AUG}
For any integers $m\ge 2$ and positive integers $p_{1}, \ldots, p_{m} $, the universal Gr\"obner basis  for  the toric ideal $I_{A}$   is given by  $$\mathcal G=\{e'_{i}e''_{j}-e''_{i}e'_{j}\ |\ 1 \leq i < j \leq m\}.$$
\end{Lemma}

\begin{proof} It follows from \cite[Lemma~3.2]{OH} together with \cite[Proposition~10.1.10]{V1}.
\end{proof}

Going forward, we work in the standard graded polynomial ring $$\mathbb{K}[E(A)] = \mathbb{K}[e_{1,1},\ldots,e_{1,2p_{1}+1},\ldots\ldots, e_{m,1},\ldots,e_{m,2p_{m}+1}].$$
Let $<$ denote the lexicographic monomial order on  $\mathbb{K}[E(A)]$  satisfying
\begin{equation*}\label{monomialorder1}
e_{1,1}<\cdots <e_{1,2p_{1}+1}<\cdots\cdots <e_{m,1}<\cdots <e_{m,2p_{m}+1},
\end{equation*}
and let $J_A$ denote the initial ideal of $I_A$ with respect to the monomial order $<$.
\begin{Proposition}
 The minimal set of monomial generators of $J_A$ is given by $$\mathcal{M}=\left\{e''_{i}e'_{j}\ |\   1 \leq i < j \leq m \right\}.$$
\end{Proposition}
\begin{proof}
Note that $e'_{i}e''_{j}<e''_{i}e'_{j}$ for $1\leq i<j\leq m$, we can deduce  from Lemma~\ref{AUG} that $J_A$ is generated by $\mathcal{M}$. The minimality of $\mathcal{M}$ can be checked directly.
\end{proof}

\subsection{Compact graphs of type 2}  \label{3.2} Assume that  $n,m$ and $p_{1}, \ldots, p_{m}, q_{1}, \ldots, q_{n}$ are given positive integers. Let $s\geq 0$ be an even number.
 We use $B$ denote the graph $B^{s}_{p_{1}, \ldots, p_{m}:q_{1}, \ldots, q_{n}}$ for short. Then, we may assume that $B$ has  vertex set
\begin{align*}
V(B)=\{& u,v\}  \cup \{w_{1}, \ldots, w_{s-1}\} \\
\cup \{u_{i,j}\mid 1 \leq i \leq m, 1 \leq j \leq 2p_{i}\} & \cup \{v_{i,j}\mid 1 \leq i \leq n, 1 \leq j \leq 2q_{i}\}
\end{align*}

and edge set
\begin{align*}
E(B)=\{& \{u_{i,j},u_{i,j+1}\} \mid 1 \leq i \leq m, 1 \leq j \leq 2p_{i}-1 \}\\
&\cup \{ \{u,u_{i,1}\}, \{u,u_{i,2p_{i}}\} \mid  1 \leq i \leq m \} \\
&\cup \{\{u,v\},\{u,w_{1}\},\{v,w_{s-1}\}\} \cup \{\{w_{i},w_{i+1}\} \mid  1 \leq i \leq s-2\} \\
&\cup \{ \{v_{i,j},v_{i,j+1}\} \mid 1 \leq i \leq n, 1 \leq j \leq 2q_{i}-1 \}\\
&\cup \{ \{v,v_{i,1}\}, \{v,v_{i,2q_{i}}\} \mid  1 \leq i \leq n \}.
\end{align*}

The edges of $B$ are labeled as follows.
 For $i\in{\{1,\ldots,m\}}$ let $e_{i,1}=\{u,u_{i,1}\}$ and $e_{i,2p_{i}+1}=\{u, u_{i,2p_i}\}$.
 For $i\in{\{1,\ldots,m\}}$ and $j\in{\{1,\ldots,2p_{i}-1\}}$ let $e_{i,j+1}=\{u_{i,j},u_{i,j+1}\}$ .
 Let $x=\{u,v \}$, $x_{1}=\{u,w_{1} \}$ and $x_{s}=\{v,w_{s-1}\}$. For $i\in{\{1,\ldots,s-2\}}$ let $x_{i+1}=\{w_{i},w_{i+1}\}$.
 For $i\in{\{1,\ldots,n\}}$ let $f_{i,1}=\{v,v_{i,1}\}$ and $f_{i,2q_{i}+1}=\{v,v_{i,2q_{i}+1}\}$.
For $i\in{\{1,\ldots,n\}}$ and $j\in{\{1,\ldots,2q_{i}-1\}}$ let $f_{i,j+1}=\{v_{i,j},v_{i,j+1}\}$.

We put $$e_i'=e_{i,1}e_{i,3}\cdots e_{i,2p_i+1}  \mbox{ and } e_i''=e_{i,2}e_{i,4}\cdots e_{i,2p_i},$$
$$f_i'=f_{i,1}f_{i,3}\cdots f_{i,2q_i+1}  \mbox{ and } f_i''=f_{i,2}f_{i,4}\cdots f_{i,2q_i},$$
and put $$x'=x_1x_3\cdots x_{s-1}, \mbox{ and } x''=x_2x_4\cdots x_{s}.$$

Note that if $s=0$ then both $x'$ and $x''$ vanish.
\begin{Lemma}\label{BUG}
For any positive integers $m,n$ and $p_{1}, \ldots, p_{m}, q_{1}, \ldots, q_{n} $ and an integer $s\geq 0$, the universal Gr\"obner basis of $I_B$  is given by  $\mathcal G=\mathcal G_1\cup \mathcal G_2\cup \mathcal G_3 \cup \mathcal G_4\cup \mathcal G_5\cup \mathcal G_6$,  where
\begin{enumerate}
\item[$(i)$] $\mathcal G_1= \{e'_{i}e''_{j}-e''_{i}e'_{j}\ |\ 1 \leq i < j \leq m\}$,
\item[$(ii)$] $\mathcal G_2= \{f'_{i}f''_{j}-f''_{i}f'_{j}\ |\ 1 \leq i < j \leq n\}$,
\item[$(iii)$] $\mathcal G_3=\{e'_{i}f'_{j}-e''_{i}x^{2}f''_{j}\ |\ 1 \leq i \leq m ,1 \leq j \leq n\}$,
\item[$(iv)$] $\mathcal G_4=\{ e'_{i}x''^{2}f''_{j}-e''_{i}x'^{2}f'_{j}\ |\ 1 \leq i \leq m ,1 \leq j \leq n\}$,
\item[$(v)$] $\mathcal G_5=\{ e'_{i}x''-e''_{i}x'x\ |\ 1 \leq i \leq m\}$,and
\item[$(vi)$] $\mathcal G_6=\{ f'_{i}x'-f''_{i}x''x\ |\ 1 \leq i \leq n\}$.
\end{enumerate}
 It should be noted that $\mathcal G_4, \mathcal G_5$ and  $\mathcal G_6$ vanish if $s=0$.
\end{Lemma}

\begin{proof} By \cite[Lemma~5.11]{HHO}, every primitive even closed walk of $B$ is one of the followings: \begin{itemize}
    \item $(e_{i,1},\ldots,e_{i,2p_{i}+1}, e_{j,1},\ldots, e_{j,2p_{j}+1})$, where $1 \leq i < j \leq m$,
    \item $(f_{i,1},\ldots,f_{i,2q_{i}+1}, f_{j,1},\ldots, f_{j,2q_{j}+1} )$, where $1 \leq i < j \leq n$,
    \item $(e_{i,1},\ldots,e_{i,2p_{i}+1},x,f_{j,1},\ldots, f_{j,2q_{j}+1},x )$, where $1 \leq i \leq m,1 \leq j \leq n$,
    \item $(e_{i,1},\ldots,e_{i,2p_{i}+1},x_{1}, \ldots, x_{s},f_{j,1}, \ldots, f_{j,2q_{j}+1},x_{s},\ldots,x_{1} )$, where $1 \leq i \leq m, 1 \leq j \leq n$,
    \item $(e_{i,1},\ldots,e_{i,2p_{i}+1},x_{1},\ldots,x_{s},x )$, where $1 \leq i \leq m$, and
    \item $(f_{i,1},\ldots,f_{i,2q_{i}+1},x,x_{1},\ldots,x_{s} )$, where $1 \leq i \leq n$.
 \end{itemize}The result  now follows  from \cite[Proposition~10.1.10]{V1}.\end{proof}

Let $<$ denote the lexicographic monomial ordering on the polynomial ring $\mathbb{K}[E(B)]$  satisfying
\begin{equation*}\label{monomialorder2}\begin{split}
e_{1,1}<\cdots <e_{1,2p_{1}+1}<\cdots\cdots <e_{m,1}<\cdots e_{m,2p_{m}+1}<x<x_{1}<\cdots <x_{s}\\<f_{1,1}<
\cdots <f_{1,2q_{1}+1}<\cdots\cdots <f_{n,1}<\cdots < f_{n,2q_{n}+1}.\end{split}
\end{equation*}
and let $J_B$ be the initial ideal of $I_B$ with respect to this order.

\begin{Proposition}
 The minimal set of monomial generators of $J_B$ is given by $\mathcal{M}=\mathcal{M}_{1}\cup\mathcal{M}_{2}\cup\mathcal{M}_{3}\cup\mathcal{M}_{4}\cup\mathcal{M}_{5}$,  where
\begin{enumerate}
\item[$(i)$] $\mathcal M_1= \left\{e''_{i}e'_{j}\ |\   1 \leq i < j \leq m \right\}$,
\item[$(ii)$] $\mathcal M_2= \left\{f''_{i}f'_{j}\ |\   1 \leq i < j \leq n \right\}$,
\item[$(iii)$] $\mathcal M_3=\left\{ e'_{i}f'_{j}\ |\   1 \leq i \leq m,1 \leq j \leq n \right\}$,
\item[$(iv)$] $\mathcal M_4=\left\{ e'_{i}x''\ |\   1 \leq i \leq m \right\}$, and
\item[$(v)$] $\mathcal M_5=\left\{ f'_{i}x'\ |\   1 \leq i \leq n \right\}$.
\end{enumerate}
 It should be noted that $\mathcal{M}_{4}$ and $\mathcal{M}_{5}$ vanish if $s=0$.
\end{Proposition}

\begin{proof}
That $\mathcal{M}$ is a generating set with respect to the given order follows from Lemma~\ref{BUG}. That it is minimal follows from the fact that none of the monomials are divided by any of the others.
\end{proof}

\subsection{Compact graphs of type 3} \label{3.3}

 Given positive integers $m,n,k$, as well as the tuples $\underline{p}=(p_{1}, \ldots, p_{m})$, $\underline{q}=(q_{1}, \ldots, q_{n})$ and  $\underline{r}=(r_{1}, \ldots, r_{k})$,
we denote the graph $C_{\underline{p}:\underline{q}:\underline{r}}$ as $C$ for brevity. Here, $p_i,q_i,r_i$ are all positive integers. By definition, we may assume $C$ has  vertex set
\begin{align*}
V(C)=&\{\{ u,v,w\} \cup \{u_{i,j}\mid 1 \leq i \leq m, 1 \leq j \leq 2p_{i}\}\\
& \cup \{v_{i,j}\mid 1 \leq i \leq n, 1 \leq j \leq 2q_{i}\} \cup \{w_{i,j}\mid 1 \leq i \leq k, 1 \leq j \leq 2r_{i}\},
\end{align*}

and edge set
\begin{align*}
E(C)=&\{ \{u_{i,j},u_{i,j+1}\} \mid 1 \leq i \leq m, 1 \leq j \leq 2p_{i}-1 \}\\
&\cup \{ \{u,u_{i,1}\}, \{u,u_{i,2p_{i}}\} \mid  1 \leq i \leq m \} \\
&\cup \{ \{v_{i,j},v_{i,j+1}\} \mid 1 \leq i \leq n, 1 \leq j \leq 2q_{i}-1 \}\\
&\cup \{ \{v,v_{i,1}\}, \{v,v_{i,2q_{i}}\} \mid  1 \leq i \leq n \}\\
&\cup \{ \{w_{i,j},w_{i,j+1}\} \mid 1 \leq i \leq k, 1 \leq j \leq 2r_{i}-1 \}\\
&\cup \{ \{w,w_{i,1}\}, \{w,w_{i,2r_{i}}\} \mid  1 \leq i \leq k \}\\
&\cup \{\{u,v\}, \{v, w\},\{w,u\}\} .
\end{align*}

We assign labels to the edges of $C$ as follows:
For $i\in{\{1,\ldots,m\}}$, let $e_{i,1}=\{u,u_{i,1}\}$ and $e_{i,2p_{i}+1}=\{u,u_{i,2p_{i}+1}\}$.
For $i\in{\{1,\ldots,m\}}$ and $j\in{\{1,\ldots,2p_{i}-1\}}$, let $e_{i,j+1}=\{u_{i,j},u_{i,j+1}\}$.

For $i\in{\{1,\ldots,n\}}$, let $f_{i,1}=\{v,v_{i,1}\}$ and $f_{i,2q_{i}+1}=\{v,v_{i,2q_{i}+1}\}$.
For $i\in{\{1,\ldots,n\}}$ and $j\in{\{1,\ldots,2q_{i}-1\}}$, let $f_{i,j+1}=\{v_{i,j},v_{i,j+1}\}$.

For $i\in{\{1,\ldots,k\}}$, let $g_{i,1}=\{w,w_{i,1}\}$ and $g_{i,2r_{i}+1}=\{w,w_{i,2r_{i}+1}\}$.
For $i\in{\{1,\ldots,k\}}$ and $j\in{\{1,\ldots,2r_{i}-1\}}$, let $g_{i,j+1}=\{w_{i,j},w_{i,j+1}\}$.

Furthermore, we define $x=\{u,v \}$, $y=\{v,w\}$, and $z=\{w,u\}$.

\begin{Lemma}\label{CUG}
For any integers $m,n,k$ and $p_{1}, \ldots, p_{m}, q_{1}, \ldots, q_{n}, r_{1}, \ldots, r_{k}$, the universal  Gr\"obner basis of $I_{C}$ is given by  in $\mathcal G=\mathcal G_1\cup \mathcal G_2\cup \mathcal G_3 \cup \mathcal G_4\cup \mathcal G_5\cup \mathcal G_6\cup\mathcal G_7\cup \mathcal G_8\cup \mathcal G_9 \cup \mathcal G_{10}\cup \mathcal G_{11}\cup \mathcal G_{12}$, where
\begin{enumerate}
\item[$(i)$] $\mathcal G_1= \{e'_{i}e''_{j}-e''_{i}e'_{j}\ |\ 1 \leq i < j \leq m\}$,
\item[$(ii)$] $\mathcal G_2= \{f'_{i}f''_{j}-f''_{i}f'_{j}\ |\ 1 \leq i < j \leq n\}$,
\item[$(iii)$] $\mathcal G_3=\{g'_{i}g''_{j}-g''_{i}g'_{j}\ |\ 1 \leq i < j \leq k\}$,
\item[$(iv)$] $\mathcal G_4=\{e'_{i}f'_{j}-e''_{i}x^{2}f''_{j}\ |\ 1 \leq i \leq m ,1 \leq j \leq n\}$,
\item[$(v)$] $\mathcal G_5=\{f'_{i}g'_{j}-f''_{i}y^{2}g''_{j}\ |\ 1 \leq i \leq n ,1 \leq j \leq k\}$,
\item[$(vi)$] $\mathcal G_6=\{g'_{i}e'_{j}-g''_{i}z^{2}e''_{j}\ |\ 1 \leq i \leq k ,1 \leq j \leq m\}$,
\item[$(vii)$] $\mathcal G_7= \{e'_{i}y^{2}f''_{j}-e''_{i}z^{2}f'_{j}\ |\ 1 \leq i \leq m ,1 \leq j \leq n\}$,
\item[$(viii)$] $\mathcal G_8= \{f'_{i}z^{2}g''_{j}-f''_{i}x^{2}g'_{j}\ |\ 1 \leq i \leq n ,1 \leq j \leq k\}$,
\item[$(ix)$] $\mathcal G_9= \{g'_{i}x^{2}e''_{j}-g''_{i}y^{2}e'_{j}\ |\ 1 \leq i \leq k ,1 \leq j \leq m\}$,
\item[$(x)$] $\mathcal G_{10}=\{e'_{i}y-e''_{i}zx\ |\ 1 \leq i \leq m\}$,
\item[$(xi)$] $\mathcal G_{11}=\{f'_{i}z-f''_{i}xy\ |\ 1 \leq i \leq n\}$, and
\item[$(xii)$] $\mathcal G_{12}=\{g'_{i}x-g''_{i}yz\ |\ 1 \leq i \leq k\}$.
\end{enumerate}
\end{Lemma}

\begin{proof}
In view of \cite[Lemma~5.11]{HHO}, every primitive even closed walk of $C$ is one of the followings:
\begin{itemize}
    \item $(e_{i,1},\ldots,e_{i,2p_{i}+1}, e_{j,1},\ldots, e_{j,2p_{j}+1})$, where $1 \leq i < j \leq m$,
    \item $(f_{i,1},\ldots,f_{i,2q_{i}+1}, f_{j,1},\ldots, f_{j,2q_{j}+1} )$, where $1 \leq i < j \leq n$,
    \item $(g_{i,1},\ldots,g_{i,2r_{i}+1}, g_{j,1},\ldots, g_{j,2r_{j}+1} )$, where $1 \leq i < j \leq k$,
    \item $(e_{i,1},\ldots,e_{i,2p_{i}+1},x,f_{j,1},\ldots, f_{j,2q_{j}+1},x )$, where $1 \leq i \leq m,1 \leq j \leq n$,
    \item $(f_{i,1},\ldots,f_{i,2q_{i}+1},y,g_{j,1},\ldots, g_{j,2r_{j}+1},y )$, where $1 \leq i \leq n,1 \leq j \leq k$,
    \item $(g_{i,1},\ldots,g_{i,2r_{i}+1},z,e_{j,1},\ldots,e_{j,2p_{j}+1},z )$, where $1 \leq i \leq k,1 \leq j \leq m$,
   \item $(e_{i,1},\ldots,e_{i,2p_{i}+1},z,y,f_{j,1},\ldots, f_{j,2q_{j}+1},y,z)$, where $1 \leq i \leq m, 1 \leq j \leq n$,
    \item $(f_{i,1},\ldots,f_{i,2q_{i}+1},x,z,g_{j,1},\ldots, g_{j,2r_{j}+1},z,x)$, where $1 \leq i \leq n,1 \leq j \leq k$,
    \item $(g_{i,1},\ldots,g_{i,2r_{i}+1},z,x,e_{j,1},\ldots,e_{j,2p_{j}+1},x,y)$, where $1 \leq i \leq k,1 \leq j \leq m$,
    \item $(e_{i,1},\ldots,e_{i,2p_{i}+1},z,y,x)$, where $1 \leq i \leq m$,
    \item $(f_{i,1},\ldots,f_{i,2q_{i}+1},x,z,y)$, where $1 \leq i \leq n$, and
    \item $(g_{i,1},\ldots,g_{i,2r_{i}+1},y,x,z)$, where $1 \leq i \leq k$.
 \end{itemize}

Now the result  follows  from \cite[Proposition~10.1.10]{V1}.\end{proof}

Going forward, we work in the standard graded polynomial ring $\mathbb{K}[E(C)]$, where the variables (i.e., the edges of $C$) is ordered as follows:
\begin{equation*}\begin{split}
& e_{1,1}<\cdots <e_{1,2p_{1}+1}<\cdots\cdots <e_{m,1}<\cdots e_{m,2p_{m}+1}<x<z<y
 \\ & <f_{1,1}<\cdots<f_{1,2q_{1}+1}<\cdots\cdots <f_{n,1}<\cdots f_{n,2q_{n}+1}<g_{1,1}<
\cdots\\&  <g_{1,2r_{1}+1}<\cdots\cdots <g_{k,1}<\cdots <g_{k,2r_{k}+1}.\end{split}
\end{equation*}
Let $J_C$ denote the initial ideal of $I_C$ with respect to the lexicographic monomial ordering $<$ on  $\mathbb{K}[E(C)]$   induced by the above order of variables.

By putting: $$e_i'=e_{i,1}e_{i,3}\cdots e_{i,2p_i+1}  \mbox{ and }  e_i''=e_{i,2}e_{i,4}\cdots e_{i,2p_i},$$
$$f_j'=f_{j,1}f_{j,3}\cdots f_{j,2q_j+1}  \mbox{ and } f_j''=f_{j,2}f_{j,4}\cdots f_{j,2q_j},$$
and
$$g_{\ell}'=g_{\ell,1}g_{\ell,3}\cdots g_{\ell,2r_\ell+1}  \mbox{ and  } g_{\ell}''=g_{\ell,2}g_{\ell,4}\cdots g_{\ell,2r_\ell},$$
where $1\leq i\leq m, 1\leq j\leq n$ and $1\leq \ell \leq k$, we obtain the following result.
\begin{Proposition}
The minimal set of monomial generators of $J_C$ is given by $\mathcal{M}=\mathcal{M}_{1}\cup\mathcal{M}_{2}\cup\mathcal{M}_{3}\cup\mathcal{M}_{4}\cup\mathcal{M}_{5}\cup\mathcal{M}_{6}\cup\mathcal{M}_{7}\cup\mathcal{M}_{8}\cup\mathcal{M}_{9}$,  where
\begin{enumerate}
\item[$(i)$] $\mathcal M_1= \left\{e''_{i}e'_{j}\ |\   1 \leq i < j \leq m \right\}$,
\item[$(ii)$] $\mathcal M_2= \left\{f''_{i}f'_{j}\ |\   1 \leq i < j \leq n \right\}$,
\item[$(iii)$] $\mathcal M_3=\left\{g''_{i}g'_{j}\ |\   1 \leq i < j \leq k  \right\}$,
\item[$(iv)$] $\mathcal M_4=\left\{e'_{i}f'_{j}\ |\   1 \leq i \leq m,1 \leq j \leq n \right\}$,
\item[$(v)$] $\mathcal M_5=\left\{f'_{i}g'_{j}\ |\   1 \leq i \leq n,1 \leq j \leq k \right\}$,
\item[$(vi)$] $\mathcal M_6=\left\{g'_{i}e'_{j}\ |\   1 \leq i \leq k,1 \leq j \leq m \right\}$,
\item[$(vii)$] $\mathcal M_7=\left\{e'_{i}y\ |\   1 \leq i \leq m \right\}$,
\item[$(viii)$] $\mathcal M_8=\left\{f'_{i}z\ |\   1 \leq i \leq n \right\}$, and
\item[$(ix)$] $\mathcal M_9=\left\{ g'_{i}x\ |\   1 \leq i \leq k \right\}$.
\end{enumerate}
\end{Proposition}
\begin{proof}
That $\mathcal{M}$ is a generating set with respect to the given order follows from Lemma~\ref{CUG}. That it is minimal follows from the fact that none of the monomials are divided by any of the others.
\end{proof}

\section{Projective dimension and regularity}

In this section, we aim to establish the following results. For convenience, we denote by $\mathfrak{t}(G)$ the number of induced cycles of $G$.
\begin{Theorem} \label{main4}
Let $G$ be a compact graph. Then  there is a monomial order $<$ such  that \begin{enumerate} \item  $\beta_{i}(\mathrm{in} _{<}(I_{G}))=(i+1)\binom{\mathfrak{t}(G)}{i+2}$ for all $i\geq 0$;

\item $\mathrm{pdim}(\mathbb{K}[E(G)]/\mathrm{in} _{<}(I_{G}))=\mathfrak{t}(G)-1$;
    \item $\mathbb{K}[E(G)]/\mathrm{in} _{<}(I_{G})$ is a Cohen-Macaulay ring;
    \item $\mathrm{reg}(\mathbb{K}[E(G)]/\mathrm{in} _{<}(I_{G}))=\mathrm{mat}(G)$.
     \end{enumerate}
\end{Theorem}

    \begin{Corollary} \label{main41} Let Let $G$ be a compact graph. Then \begin{enumerate}
  \item  $\mathrm{pdim}(\mathbb{K}[G])=\mathfrak{t}(G)-1$;
  \item  $\mathrm{reg}(\mathbb{K}[G])=\mathrm{mat}(G)$.
\end{enumerate}
\end{Corollary}

It is known that if $I$ is a graded ideal of a polynomial ring $R$ such that $R/\mathrm{in}_{<}(I)$ is Cohen-Macaulay for some monomial order $<$, then $R/I$ is also Cohen-Macaulay. Furthermore, we have $\mathrm{reg}(R/I)=\mathrm{reg}(R/\mathrm{in}_{<}(I))$ and $\mathrm{pdim}(R/I)=\mathrm{pdim}(R/\mathrm{in}_{<}(I))$. Based on these facts we could see that Corollary~\ref{main41} follows immediately from Theorem~\ref{main4}. Regarding the proof of Theorem~\ref{main4}, we will provide it at the end of this section.

 Assume through this section that $m,n,k$ are positive integers and  $\underline{p}=(p_{1}, \ldots, p_{m})$, $\underline{q}=(q_{1}, \ldots, q_{n})$ and  $\underline{r}=(r_{1}, \ldots, r_{k})$ are integral tuples with positive entries. Also, we write $\underline{p'}=(p_{1}, \ldots, p_{m-1})$, $\underline{q'}=(q_{1}, \ldots, q_{n-1})$, and $\underline{r'}=(r_{1}, \ldots, r_{k-1})$.

\subsection{type one}   In this subsection, we always use the monomial order given in Subsection~\ref{3.1}, and denote the toric ideal of $A_{\underline{p}}$ and its initial ideal  as $I_m$ and $J_m$, respectively. Similarly, $I_{m-1}$ and $J_{m-1}$ represent the toric ideal of $A_{\underline{p'}}$ and its initial ideal   respectively.  Recall from Subsection~\ref{3.1} that $G(J_{m})=\left\{e''_{i}e'_{j}\ |\   1 \leq i < j \leq m \right\}$.

\begin{Proposition}\label{AEK} Denote by $H_{m}$ the monomial ideal $(e''_{1}, \ldots, e''_{m-1})$.
Then $J_{m}=J_{m-1}+e'_{m}H_{m}$ is an E-K splitting. Furthermore $J_{m-1} \cap  e'_{m}H_{m}=e'_{m}J_{m-1}$.

\end{Proposition}

\begin{proof}
First of all, it is easy to see that $G(J_{m})=G(J_{m-1})\sqcup G(e'_{m}H_{m})$. Let us check $J_{m-1} \cap  e'_{m}H_{m}=e'_{m}J_{m-1}$. Take any $e_i''e_j'\in G(J_{m-1})$, since $1\leq i<j\leq m-1$, we have $e_m'e_i''e_j'\in e'_{m}H_{m}\cap J_{m-1}$. For the converse, take $e'_me''_{i_1}\in G(e'_mH_{m})$ and $e''_{i_2}e'_j\in G(J_{m-1})$. Here, $1\leq i_1\leq m-1$ and $1\leq i_2<j\leq m-1$. Then, $$\mathrm{lcm}(e'_me''_{i_1}, e''_{i_2}e'_j)\in (e'_me''_{i_2}e'_j)\subseteq e'_mJ_{m-1}.$$
This shows $J_{m-1} \cap  e'_{m}H_{m}=e'_{m}J_{m-1}$.

Next, we define functions $\phi$ and $\psi$ as follows:
\[ \phi: G(e'_{m}J_{m-1}) \to G(J_{m-1}),\quad  e'_{m}e''_{i}e'_{j} \mapsto e''_{i}e'_{j}, \quad 1\leq i<j \leq m-1, \]
\[ \psi: G(e'_{m}J_{m-1}) \to  G(e'_{m}H_{m}), \quad  e'_{m}e''_{i}e'_{j} \mapsto e'_{m}e''_{i}, \quad 1\leq i<j \leq m-1.  \]

(1)  Let $u$ be  a minimal generator of $e_m'J_{m-1}$. Then  $u=e'_{m}e''_{i}e'_{j}$ for some $1\leq i<j \leq m-1$. It follows that
$$\mathrm{lcm}(\phi(u),\psi(u))=\mathrm{lcm}(e''_{i}e'_{j}, e'_{m}e''_{i})=e'_{m}e''_{i}e'_{j}=u.$$

(2) Let $C=e'_{m}(e''_{i_{1}}e'_{j_{1}}, \ldots, e''_{i_{k}}e'_{j_{k}})$ be a non-empty subset of $G(e'_{m}J_{m-1})$, where $1\leq i_{q}<j_{q} \leq m-1$ for $q=1,\ldots,k$. Then
$$\phi(C)=(e''_{i_{1}}e'_{j_{1}}, \ldots, e''_{i_{k}}e'_{j_{k}}) \mbox{ and  } \psi(C)=e'_{m}(e''_{i_{1}}, \ldots, e''_{i_{k}}).$$

Since $e_i''$ and $e_j'$ are co-prime for all $1\leq i,j\leq m$, we have $$\mathrm{lcm}(C)=e'_{m}\mathrm{lcm}(\phi(C)) \mbox{ and } \mathrm{lcm}(C)=\mathrm{lcm}(\psi(C))\mathrm{lcm}(e'_{j_{1}},\ldots,e'_{j_{k}}).$$
This completes the proof.
\end{proof}

In the following, we will utilize the following formula without explicitly referencing it:  For a finitely generated graded module $M$ over a standard graded polynomial ring, one has
$$\max\{j-i\mid \beta_{i-a,j-b}(M)\neq 0\}=\max\{\ell+b-(k+a)\mid \beta_{k,\ell}(M)\neq 0\}=\reg(M)+b-a.$$

\begin{Proposition}\label{reg1} Let $m\geq 2$. Then $\reg(J_m)=\mathrm{mat}(A_{\underline{p}})+1$.
\end{Proposition}
\begin{proof}
It is easy to check that  $\mathrm{mat}(A_{\underline{p}})=\sum\limits_{i=1}^{m}p_{i}$.
We proceed with the induction on $m$. If $m=2$, since $J_2$ is generated by a single monomial of degree $p_1+p_2+1$, we obtain $\reg(J_2)=p_1+p_2+1$.

Suppose that $m> 2$. Then,
by Lemma~\ref{AEK}, we have
\begin{equation}\label{ekgraded}\tag{$\clubsuit$}
\beta_{i,j}(J_m)=\beta_{i,j}(J_{m-1})+\beta_{i, j-p_m-1}(H_m)+\beta_{i-1,j-p_m-1}(J_{m-1}).
\end{equation}  It follows that
\begin{equation*} \begin{split} \reg (J_m)&=\max\{j-i \mid \beta_{i,j}(J_{m-1})+ \beta_{i-1,j-p_m-1}(J_{m-1}) + \beta_{i,j-p_m-1}(H_m)\neq 0\}\\ &=\max\{\reg(J_{m-1}), \reg(J_{m-1})+p_m, \reg(H_m)+p_m+1\}.\end{split}
\end{equation*}
Note that $H_m$ is generated by a regular sequence of degrees $p_1,\ldots, p_{m-1}$. By using the Koszul theory,  we obtain $\reg(H_m)=\sum\limits_{i=1}^{m-1}p_i-m+2$. Hence, \begin{equation*}\begin{split}\reg(J_m)&=\max\{\sum\limits_{i=1}^{m}p_i+1,\sum\limits_{i=1}^{m}p_i-m+3\}\\&=\sum\limits_{i=1}^{m}p_i+1,  \end{split}
\end{equation*} as desired.
\end{proof}

\begin{Proposition}\label{total1} Let $m\geq 2$. Then $\beta_i(J_m)=(i+1)\binom{m}{i+2}$ for all $i\geq 0$. In particular, $\mathrm{pdim}(J_m)=m-2$.
\end{Proposition}
\begin{proof} We also employ the induction on $m$. The case that $m=2$ or $i=0$ is straightforward. If $m\geq 3$ and $i\geq 1$, then, by noting that the formula $\binom{m}{i}=\binom{m-1}{i-1}+\binom{m-1}{i}$ holds for all $m\geq 1$ and $i\geq 1$, we have
\begin{equation*}\begin{split} \beta_{i}(J_m)&=\beta_{i}(J_{m-1})+\beta_i(H_m)+\beta_{i-1}(J_{m-1})
\\&=(i+1)\binom{m-1}{i+2}+\binom{m-1}{i+1}+i\binom{m-1}{i+1}\\&=(i+1)\binom{m}{i+2},
\end{split}\end{equation*} as desired.
\end{proof}

We may compute the graded Betti numbers of $J_m$ in a special case.

\begin{Proposition} \label{Betti}If $p_1=\cdots=p_m=p$, then for all $i\geq 0$, we have  $$\beta_{i,j}(J_m)=\left\{
                                                                   \begin{array}{ll}
                                                                     \binom{m}{i+2}, & \hbox{$j=(i+2)p+\ell, \ell=1,\ldots,i+1$;} \\
                                                                     0, & \hbox{otherwise.}
                                                                   \end{array}
                                                                 \right.
$$
\end{Proposition}
\begin{proof} We use the induction on $m$. The case that $m=2$ or $i=0$ are obvious.
So we suppose $m\geq3$ and $i\geq1$. By the induction hypothesis we have
$$\beta_{i,j}(J_{m-1})=\left\{
                                                                   \begin{array}{ll}
                                                                     \binom{m-1}{i+2}, & \hbox{$j=(i+2)p+\ell, \ell=1,\ldots,i+1$;} \\
                                                                     0, & \hbox{otherwise}
                                                                   \end{array}
                                                                 \right.
$$
and so
$$\beta_{i-1,j-p-1}(J_{m-1})=\left\{
                                                                   \begin{array}{ll}
                                                                     \binom{m-1}{i+1}, & \hbox{$j=(i+2)p+\ell+1, \ell=1,\ldots,i$;} \\
                                                                     0, & \hbox{otherwise.}
                                                                   \end{array}
                                                                 \right.
$$

  On the other hand, by the theory of Koszul complex,  we have $$\beta_{i,j-p-1}(H_m)=\left\{
                                             \begin{array}{ll}
                                               \binom{m-1}{i+1}, & \hbox{$j=(i+2)p+1$;} \\
                                               0, & \hbox{otherwise.}
                                             \end{array}
                                           \right.
$$

Now, the result follows by applying the equality~($\clubsuit$).
\end{proof}

\subsection{Type two }In this subsection, we denote the toric ideal of $B^{s}_{\underline{p}:\underline{q}}$ and its initial ideal   as $I^{s}_{{m:n}}$ and $J^{s}_{{m:n}}$   respectively. Similarly, $I^{s}_{{m:n-1}}$ and $J^{s}_{{m:n-1}}$ represent the toric ideal and its initial ideal  of $B^{s}_{\underline{p}:\underline{q'}}$ respectively. Here, we  use the monomial order given in Subsection~\ref{3.2}. The distinction between the case when $s>0$ and the case when $s=0$ is significant. Let us first consider the case when $s>0$. Recall from Subsection~\ref{3.2} that $G(J_{m:n}^s)$ is the set $\left\{e''_{i}e'_{j}\ |\   1 \leq i < j \leq m \right\}\sqcup \left\{f''_{i}f'_{j}\ |\   1 \leq i < j \leq n \right\} \sqcup \left\{ e'_{i}f'_{j}\ |\   1 \leq i \leq m,1 \leq j \leq n \right\} \sqcup \left\{ e'_{i}x''\ |\   1 \leq i \leq m \right\}\sqcup \left\{ f'_{i}x'\ |\   1 \leq i \leq n \right\}.$

\begin{Proposition}\label{BSEKSP} Denote by $H^{s}_{m:n}$ the monomial ideal $(e'_{1}, \ldots, e'_{m},f''_{1}, \ldots, f''_{n-1},x')$. Then
$J^{s}_{m:n}=J^{s}_{m:n-1}+f'_{n}H^{s}_{m:n}$ is an E-K splitting. Furthermore, $J^{s}_{m:n-1} \cap f'_{n}H^{s}_{m:n}=f'_{n}J^{s}_{m:n-1}$.

\end{Proposition}

\begin{proof}

First of all, it is routine  to see that $G(J^{s}_{m:n})=G(J^{s}_{m:n-1})\bigsqcup G(f'_{n}H^{s}_{m:n})$ and $J^{s}_{m:n-1} \cap f'_{n}H^{s}_{m:n}=f'_{n}J^{s}_{m:n-1}$.

Let us define a function $\phi :G(f'_{n}J^{s}_{{m:n-1}}) \to G(J^{s}_{{m:n-1}})$ that sends $f'_{n}u$ to $u$ for all $u\in G(J^{s}_{{m:n-1}})$.

Similarly, we define a function $\psi :G(f'_{n}J^{s}_{{m:n-1}}) \to  G(f'_{n}H^{s}_{m:n})$ using the following rules:
\begin{itemize}
  \item $f'_{n}e''_{i}e'_{j} \mapsto f'_{n}e'_{j}$ for all $1\leq i<j\leq m$ and $f'_{n}f''_{i}f'_{j} \mapsto f'_{n}f''_{i}$ for all $1\leq i<j\leq n-1$;
  \item $f'_{n}e'_{i}f'_{j} \mapsto f'_{n}e'_{i}$ for all $1\leq i\leq m$ and $1\leq j\leq n-1$;
  \item $f'_{n}e'_{i}x'' \mapsto f'_{n}e'_{i}$ for all $1\leq i \leq m$ and $f'_{n}f'_{i}x' \mapsto f'_{n}x'$ for all $1\leq i\leq n-1$.
\end{itemize}

It is routine to check that conditions (1) and (2) of Definition~\ref{DEKSP} are satisfied, thus confirming that it is indeed an E-K splitting.
\end{proof}

If $s=0$, then $G(J_{m:n}^0)$ is the disjoint union $\left\{ e'_{i}f'_{j}\ |\   1 \leq i \leq m,1 \leq j \leq n \right\} \cup \left\{e''_{i}e'_{j}\ |\   1 \leq i < j \leq m \right\}\cup \left\{f''_{i}f'_{j}\ |\   1 \leq i < j \leq n \right\}$. Similarly, we obtain the following.

\begin{Proposition}\label{BoEKSP}Denote by $H_{m:n}^0$ the monomial ideal $(e'_{1}, \ldots, e'_{m},f''_{1}, \ldots, f''_{n-1})$. Then
$J^{0}_{m:n}=J^{0}_{m:n-1}+f'_{n}H^0_{m:n}$ is an E-K splitting, and $J^{0}_{m:n-1}\cap f'_{n}H^0_{m:n}=f'_{n}J^{0}_{m:n}$.

\end{Proposition}

\begin{Proposition} \label{total2} Let $s\geq 0$ be an even number, $m,n\geq1$. Then for all $i\geq 0$, we have
$$\beta_{i}(J^{s}_{m:n})=\left\{
                                                                   \begin{array}{ll}
                                                                     (i+1)\binom{m+n}{i+2}, & \hbox{$s=0$;} \\
                                                                     (i+1)\binom{m+n+1}{i+2}, & \hbox{$s>0$}.
                                                                   \end{array}
                                                                 \right.
$$

\end{Proposition}
\begin{proof} We consider the following two cases.

\hspace{-0.5cm}$\textit{Case s=0:}$
We also employ the induction on $n$. The case that $m=n=1$ or $i=0$ is straightforward. If $m+n\geq 3$ and $i\geq 1$, then, we have

\begin{equation*}\begin{split} \beta_{i}(J^0_{m:n})&=\beta_{i}(J^0_{m:n-1})+\beta_i(H^0_{m:n})+\beta_{i-1}(J^0_{m:n-1})
\\&=(i+1)\binom{m+n-1}{i+2}+i\binom{m+n-1}{i+1}+\binom{m+n-1}{i+1}\\&=(i+1)\binom{m+n}{i+2},
\end{split}\end{equation*} as desired.

\noindent{\it Case s}$>0$:
We also employ the induction on $n$. The case that $i=0$ is straightforward. If $i\geq 1$,  we have
\begin{equation*}\begin{split} \beta_{i}(J^{s}_{m:n})&=\beta_{i}(J^{s}_{m:n-1})+\beta_i(H^{s}_{m:n})+\beta_{i-1}(J^{s}_{m:n-1})
\\&=(i+1)\binom{m+n}{i+2}+i\binom{m+n}{i+1}+\binom{m+n}{i+1}\\&=(i+1)\binom{m+n+1}{i+2}.
\end{split}\end{equation*} This completes the proof.
\end{proof}

\begin{Proposition} \label{reg2} Let $s\geq 0$ be an even number and  $m,n\geq 1$. Then $$\reg(J^{s}_{m:n})=\mathrm{mat}(B^{s}_{\underline{p}:\underline{q}})+1.$$
\end{Proposition}
\begin{proof}

 \noindent $\textit{Case s=0:}$ In this case, $\mathrm{mat}(B^0_{\underline{p}:\underline{q}})=\sum\limits_{i=1}^{m}p_{i}+\sum\limits_{i=1}^{n}q_{i}+1.$
  We proceed with the induction on $n$.
By  Proposition~\ref{BoEKSP}, we have $$ \beta_{i,j}(J^0_{m:n})=\beta_{i,j}(J^0_{m:n-1})+\beta_{i, j-q_n-1}(H^0_{m:n})+\beta_{i-1,j-q_n-1}(J^0_{m:n-1}).$$

It follows that
\begin{equation*} \begin{split} \reg (J^0_{m:n})&=\max\{j-i \mid \beta_{i,j}(J^0_{m:n-1})+\beta_{i-1,j-q_n-1}(J^{0}_{m:n-1})+\beta_{i, j-q_n-1}(H^0_{m:n})\neq 0\}\\ &=\max\{\reg(J^0_{m:n-1}), \reg(J^0_{m:n-1})+q_n, \reg(H^0_{m:n})+q_n+1\}.\end{split}
\end{equation*}
Note that $H^0_{m:n}$ is generated by a regular sequence of degrees $p_1+1,\ldots, p_{m}+1,q_1,\ldots, q_{n-1}$. By using the Koszul theory,  we obtain $$\reg(H^0_{m:n})=\sum\limits_{i=1}^{m}p_{i}+\sum\limits_{i=1}^{n-1}q_{i}-n+2.$$ Hence,
if $n=1$, since $\mathrm{reg}(J^{0}_{m:0})=\mathrm{reg}(J_{m})=\sum\limits_{i=1}^{m}p_{i}+1,$ we have $$\mathrm{reg}(J^0_{m:1})=\max\{\sum\limits_{i=1}^{m}p_{i}+q_1+1, \sum\limits_{i=1}^{m}p_{i}+q_1+2\}=\sum\limits_{i=1}^{m}p_{i}+q_1+2.$$
This proves the case when $n=1$.

If $n>1$, then \begin{equation*}\begin{split}\reg(J^0_{m:n})&=\max\{\sum\limits_{i=1}^{m}p_{i}+\sum\limits_{i=1}^{n}q_{i}+2, \sum\limits_{i=1}^{m}p_{i}+\sum\limits_{i=1}^{n}q_{i}-n+3\}\\&=\sum\limits_{i=1}^{m}p_{i}+\sum\limits_{i=1}^{n}q_{i}+2.  \end{split}
\end{equation*}

\noindent{\it Case $s>0$:} In this case, we have $\mathrm{mat}(B^{s}_{\underline{p}:\underline{q}})=\sum\limits_{i=1}^{m}p_{i}+\sum\limits_{i=1}^{n}q_{i}+\frac{s}{2}.$
We proceed with the induction on $n$ again. First, note that $$\mathrm{reg}(H_{m:n}^s)=\sum\limits_{i=1}^{m}p_{i}+\sum\limits_{i=1}^{n-1}q_{i}+\frac{s}{2}-n+1.$$

If $n=1$, then
$$ \beta_{i,j}(J^{s}_{m:1})=\beta_{i,j}(J^{s}_{m:0})+\beta_{i, j-q_1-1}(H^{s}_{m:1})+\beta_{i-1,j-q_1-1}(J^{s}_{m:0}).$$
 Note that $G(J_{m:0}^s)=\{e''_ie'_j\mid 1\leq i<j\leq m\}\cup\{e'_ix''\mid 1\leq i\leq m\}$. By putting $e''_0=x''$, we may write $G(J_{m:0}^s)=\{e''_ie'_j\mid 0\leq i<j\leq m\}.$ This is exactly the ideal studied in Subsection 4.1, and so
 it follows  from Proposition~\ref{reg1} that $$\mathrm{reg}(J_{m:0}^s)=\sum\limits_{i=1}^mp_i+\frac{s}{2}+1.$$
  Hence, $\reg(J^{s}_{m:1})=\max\{ \reg(J_{m:0}^s)+q_1, \reg(H^{s}_{m:1})+q_1+1\}=\sum\limits_{i=1}^{m}p_{i}+q_{1}+\frac{s}{2}+1$.

 Suppose that $n> 1$. Then,
since $$ \beta_{i,j}(J^{s}_{m:n})=\beta_{i,j}(J^{s}_{m:n-1})+\beta_{i, j-q_n-1}(H^{s}_{m:n})+\beta_{i-1,j-q_n-1}(J^{s}_{m:n-1}),$$
we have
\begin{equation*} \begin{split} \reg(J^{s}_{m:n}) &=\max\{\reg(J^{s}_{m:n-1}), \reg(J^{s}_{m:n-1})+q_n, \reg(H^{s}_{m:n})+q_n+1\}\\ &=\max\{\sum\limits_{i=1}^{m}p_{i}+\sum\limits_{i=1}^{n}q_{i}+\frac{s}{2}+1, \sum\limits_{i=1}^{m}p_{i}+\sum\limits_{i=1}^{n}q_{i}-n+\frac{s}{2}+2\}\\&=\sum\limits_{i=1}^{m}p_{i}+\sum\limits_{i=1}^{n}q_{i}+\frac{s}{2}+1,\end{split}
\end{equation*}
 as desired.
\end{proof}

\subsection{Type three }In this subsection, we denote the toric ideal of $C_{\underline{p}:\underline{q}:\underline{r}}$ and its initial ideal  as $I_{m:n:k}$ and $J_{m:n:k}$   respectively. Similarly, $I_{m:n:k-1}$ and $J_{m:n:k-1}$ represent the toric ideal of $C_{\underline{p}:\underline{q}:\underline{r'}}$ and its initial ideal, respectively.   Recall from Subsection~\ref{3.3} that $G(J_{m:n:k})$ is the disjoint union $\left\{e''_{i}e'_{j}\ |\   1 \leq i < j \leq m \right\}\cup\left\{f''_{i}f'_{j}\ |\   1 \leq i < j \leq n\right\}\cup\left\{g''_{i}g'_{j}\ |\   1 \leq i < j \leq k \right\}\cup\left\{e'_{i}y\ |\   1 \leq i \leq m \right\}\cup\left\{f'_{i}z\ |\   1 \leq i \leq n \right\}\cup\left\{ g'_{i}x\ |\   1 \leq i \leq k \right\}\cup\left\{e'_{i}f'_{j}\ |\   1 \leq i \leq m,1 \leq j \leq n \right\}\cup\left\{f'_{i}g'_{j}\ |\   1 \leq i \leq n,1 \leq j \leq k \right\}$

\noindent $\cup\left\{g'_{i}e'_{j}\ |\   1 \leq i \leq k,1 \leq j \leq m \right\}$.

\begin{Proposition}\label{CEKSP}Denote by $H_{m:n:k}$ the monomial ideal $$(x,e'_{1}, \ldots, e'_{m},f'_{1}, \ldots, f'_{n},g''_{1}, \ldots, g''_{k-1}).$$ Then
$J_{m:n:k}=J_{m:n:k-1}+g'_{k}H_{m:n:k}$ is an E-K splitting. Furthermore, we have $$J_{m:n:k-1} \cap g'_{k}H_{m:n:k}=g'_{k}J_{m:n:k-1}.$$

\end{Proposition}

\begin{proof}
First of all, it is routine to check that $G(J_{m:n:k})=G(J_{m:n:k-1})\bigsqcup G(g'_{k}H_{m:n:k})$ and $J_{m:n:k-1} \cap g'_{k}H_{m:n:k}=g'_{k}J_{m:n:k-1}$.

Define a function $\phi :G(g'_{k}J_{m:n:k-1}) \to G(J_{C_{m:n:k-1}})$ that sends $g'_{k}u$  to $u$ for all $ u\in G(J_{m:n:k-1}).$

Define a function $\psi :G(g'_{k}J_{m:n:k-1}) \to  G(g'_{k}H_{m:n:k})$ by the following rules:
\begin{itemize}
  \item $g'_{k}e''_{i}e'_{j} \mapsto g'_{k}e'_{j}$ for all $1\leq i<j \leq m$ and $g'_{k}e'_{i}y \mapsto g'_{k}e'_{i}$ for all $1\leq i\leq m$;
  \item $g'_{k}e'_{i}f'_{j} \mapsto g'_{k}e'_{i}$ for all $1\leq i\leq m$ and $1\leq j\leq n$;
  \item $g'_{k}f''_{i}f'_{j} \mapsto g'_{k}f'_{j}$ for all $1\leq i<j \leq n$ and $g'_{k}f'_{i}z \mapsto g'_{k}f'_{i}$  for all $1\leq i\leq n$;
  \item $g'_{k}f'_{i}g'_{j} \mapsto g'_{k}f'_{i}$ for all $1\leq i\leq k-1$ and $1\leq j\leq n$;
  \item $g'_{k}g''_{i}g'_{j} \mapsto g'_{k}g''_{i}$ for all $1\leq i<j \leq k-1$ and $g'_{k}g'_{i}x \mapsto g'_{k}x$ for all $1\leq i<j \leq k-1$;
  \item $g'_{k}g'_{i}e'_{j} \mapsto g'_{k}e'_{j}$ for all $1\leq i\leq k-1$ and $1\leq j\leq m$.
\end{itemize}

It is routine to check conditions (1) and (2) of Definition \ref{DEKSP} are satisfied.
\end{proof}

\begin{Proposition}\label{total3} Let $m,n,k\geq 1$. Then $\beta_i(J_{m:n:k})=(i+1)\binom{m+n+k+1}{i+2}$ for all $i\geq 0$. In particular,
$\mathrm{pdim}(J_{m:n:k})=m+n+k-1$.

\end{Proposition}
\begin{proof} We also employ the induction on $k$. The case that $i=0$ is straightforward. If $i\geq 1$, then,  we have
\begin{equation*}\begin{split} \beta_{i}(J_{m:n:k})&=\beta_{i}(J_{m:n:k-1})+\beta_i(H_{m:n:k})+\beta_{i-1}(J_{m:n:k-1})
\\&=(i+1)\binom{m+n+k}{i+2}+\binom{m+n+k}{i+1}+i\binom{m+n+k}{i+1}\\&=(i+1)\binom{m+n+k+1}{i+2},
\end{split} \end{equation*} as desired.
\end{proof}

\begin{Proposition}\label{reg3}  Let $m,n,k\geq 1$. Then $\reg(J_{m:n:k})=\mathrm{mat}(C_{\underline{p}:\underline{q}:\underline{r}})+1$.
\end{Proposition}
\begin{proof} Note that
$\mathrm{mat}(C_{\underline{p}:\underline{q}:\underline{r}})=\sum\limits_{i=1}^{m}p_{i}+\sum\limits_{i=1}^{n}q_{i}+\sum\limits_{i=1}^{k}r_{i}+1.$
We proceed with the induction on $k$. If $k=1$, then   $$\beta_{i,j}(J_{m:n:1})=\beta_{i,j}(J_{m:n:0})+\beta_{i, j-r_1-1}(H_{m:n:1})+\beta_{i-1,j-r_1-1}(J_{m:n:0}).$$
Since $\reg(J_{m:n:0})=\reg(J^{2}_{m:n})$,
 we obtain $\reg(J_{m:n:1})=\sum\limits_{i=1}^{m}p_{i}+\sum\limits_{i=1}^{n}q_{i}+r_{1}+2$.

 Suppose that $k> 1$. Then,
by Proposition~\ref{CEKSP}, we have $$ \beta_{i,j}(J_{m:n:k})=\beta_{i,j}(J_{m:n:k-1})+\beta_{i, j-r_k-1}(H_{m:n:k})+\beta_{i-1,j-r_k-1}(J_{m:n:k-1}).$$

It follows that $\reg (J_{m:n:k})$
\begin{equation*} \begin{split}
&=\max\{j-i \mid \beta_{i,j}(J_{m:n:k-1})+\beta_{i-1,j-r_k-1}(J_{m:n:k-1})+\beta_{i, j-r_k-1}(H_{m:n:k})\neq 0\}\\ &=\max\{\reg(J_{m:n:k-1}), \reg(H_{m:n:k})+r_k+1, \reg(J_{m:n:k-1})+r_k \}.\end{split}
\end{equation*}
Note that $H_{m:n:k}$ is generated by a regular sequence of degrees $p_1+1,\ldots, p_{m}+1,q_1+1,\ldots, q_{n}+1,r_1,\ldots, r_{k-1},1$,   we obtain $\reg(H_{m:n:k})=\sum\limits_{i=1}^{m}p_{i}+\sum\limits_{i=1}^{n}q_{i}+\sum\limits_{i=1}^{k-1}r_{i}-k+2$. Hence,
\begin{equation*}\begin{split}
\reg(J_{m:n:k})&=\max\{\sum\limits_{i=1}^{m}p_{i}+\sum\limits_{i=1}^{n}q_{i}+\sum\limits_{i=1}^{k}r_{i}+2,\sum\limits_{i=1}^{m}p_{i}+\sum\limits_{i=1}^{n}q_{i}+\sum\limits_{i=1}^{k}r_{i}-k+3\}\\
&=\sum\limits_{i=1}^{m}p_{i}+\sum\limits_{i=1}^{n}q_{i}+\sum\limits_{i=1}^{k}r_{i}+2,  \end{split}
\end{equation*}as desired.
\end{proof}

To complete the proof of Theorem~\ref{main4}, we require some additional notation and facts. Recall a connected graph is  {\it planar} if it can be drawn on a 2D plane such that none of the edges intersect. If a planar graph $G$ is drawn in this way, it divides the plane into regions called {\it faces}.
The number of faces is denoted by $f(G)$. The famous Euler formula  states that for any planar graph $G$, we have $$|E(G)|-|V(G)|=f(G)-2.$$
If we assume that every edge of $G$ belongs to at most one induced cycle, then there is a one-to-one correspondence between induced cycles and bounded faces of $G$.  Since there is exactly one unbounded face of $G$, it follows that $f(G)=\mathfrak{t}(G)+1$.

However, it is worth noting that the formula $f(G)=\mathfrak{t}(G)+1$ does not hold in general. For example, if $G$ is the complete graph with 4 vertices, then $G$ is planar, but $f(G)=\mathfrak{t}(G)=4$.

We are now ready to present the proof of Theorem~\ref{main4}.
\begin{proof}

(1) This is a combination of Propositions~\ref{total1}, \ref{total2} and Proposition~\ref{total3}.

(2) It follows immediately from (1).

(3)  Since $G$ is a compact graph, $G$ is a planar graph and every edge of $G$ belongs to at most one induced cycle. Hence, because of the discussion above, we have $$|E(G)|-|V(G)|=\mathfrak{t}(G)-1.$$
This implies \begin{equation*}\begin{split}
\mathrm{depth}(\mathbb{K}[E(G)]/\mathrm{in}_{<}(I_G))&=|E(G)|-\mathrm{pdim}(\mathbb{K}[E(G)]/\mathrm{in}_{<}(I_G))\\
&=|V(G)|=\mathrm{dim}(\mathbb{K}[G])\\&=\mathrm{dim}(\mathbb{K}[E(G)]/\mathrm{in}_{<}(I_G)).\end{split}\end{equation*}
Here, the second last equality follows from  \cite[Corollary 10.1.21]{V1}. Hence, by definition, $\mathbb{K}[E(G)]/\mathrm{in}_{<}(I_G)$ is Cohen-Macaulay.

(4) This is a combination of Propositions~\ref{reg1}, \ref{reg2} and Proposition~\ref{reg3}.
 \end{proof}

\section{Cohen-Macaulay types and top graded Betti numbers}

Assume that $G$ is a compact graph. In this section we will compute the top graded Betti numbers of $\mathbb{K}[G]$. Since $\mathbb{K}[G]$ is Cohen-Macaulay, the regularity of $\mathbb{K}[G]$ is determined by its top graded Betti numbers. Therefore, the regularity formula given in Section 4 could also be deduced from the results of this section.
To present the top graded Betti numbers of $\mathbb{K}[G]$, we need to consider three cases.
 The most complex case is when $G$  is a compact graph of type 3, and we will provide detailed proof specifically for this case. The proofs for the cases when $G$  is a compact graph of type one or type two are similar, with only minor differences, so we will only provide an outline of the proofs for those cases.

The top graded Betti numbers of the edge rings of three types of compact graphs are presented in Propositions~\ref{three}, \ref{one} and Proposition~\ref{two}, respectively. By combining the aforementioned results and their proofs, the following conclusion regarding the top total Betti numbers can be immediately derived.

\begin{Theorem}\label{5.1} Let $G$ be a compact graph, and let $I_G$ be  the toric ideal of $\mathbb{K}[G]$.  Denote by $J_G$  the initial ideal of $I_G$ with respect to the order given in Section 3. Then $I_G$ and $J_G$ share the same top graded Betti numbers. In particular, we have $\mathrm{type} (\mathbb{K}[G])=\mathfrak{t}(G)-1$.
\end{Theorem}

\subsection{type three} Let $C$ denote the compact graph $C_{\underline{p}:\underline{q}:\underline{r}}$, whose vertex set $V(C)$ and edge set $E(C)$ are given explicitly in Subsection~\ref{3.3}.  In this subsection, we compute the minimal generators of the canonical module $\omega_{\mathbb{K}[C]}$ and then determine the top graded Betti numbers of the toric ring $\mathbb{K}[C]$.

 It is easy to see that  $|V(C)|=2\sum\limits_{i=1}^{m}p_{i}+2\sum\limits_{i=1}^{n}q_{i}+2\sum\limits_{i=1}^{k}r_{i}+3$.
We use the following notions for all the  entries of $\RR^{|V(C)|}$:
\begin{align*}
\RR^{|V(C)|}=\{\sum\limits_{i=1}^{m}\sum\limits_{j=1}^{2p_{i}}{a_{i,j}\ub_{i,j}}+a\ub+
\sum\limits_{i=1}^{n}\sum\limits_{j=1}^{2q_{i}}{b_{i,j}\vb_{i,j}}+b\vb+\sum\limits_{i=1}^{k}\sum\limits_{j=1}^{2r_{i}}{c_{i,j}\wb_{i,j}}+c\wb \mid \\
a_{i,j}, a,b_{i,j}, b,c_{i,j}, c \in \RR \mbox{ for all } i,j \},
\end{align*}
where $\ub,\ub_{i,j},\vb,\vb_{i,j},\wb,\wb_{i,j}$ are the unit vectors of $\RR^{|V(C)|}$,
each $\ub_{i,j}$ (resp. $\vb_{i,j}, \wb_{i,j}$) corresponds to $u_{i,j}$ (resp. $v_{i,j}, w_{i,j}$) (where $1\leq i\leq m$ and $1\leq j\leq 2p_{i}$) (resp. $1\leq i\leq n$ and $1\leq j\leq 2q_{i}$, $1\leq i\leq k$ and $1\leq j\leq 2r_{i}$) and $\ub$ (resp. $\vb, \wb$ ) corresponds to $u$ (resp. $v,w$).

In what follows, we will construct $m+n+k$ integral vectors in $\RR^{|V(C)|}$ and then show that they are minimal vectors of $\mathrm{relint}(\RR_+(C))\cap \mathbb{Z}^{|V(C)|}$. Here, an integral vector in $\mathrm{relint}(\RR_+(C))$ is called {\it minimal} if it cannot written as the sum of a vector in $\mathrm{relint}(\RR_+(C))\cap \mathbb{Z}^{|V(C)|}$ and a nonzero vector of  $\RR_+(C)\cap \mathbb{Z}^{|V(C)|}$.
The construction is as follows:

For $\ell=1,\ldots,m$, let $$\alpha_{\ell}:=\sum\limits_{i=1}^{m}\sum\limits_{j=1}^{2p_{i}}\ub_{i,j}+\sum\limits_{i=1}^{n}\sum\limits_{j=1}^{2q_{i}}\vb_{i,j}+
\sum\limits_{i=1}^{k}\sum\limits_{j=1}^{2r_{i}}\wb_{i,j}+\vb+\wb+2{\ell}\ub.$$

For ${\ell}=1,\ldots,n$, let $$\beta_{\ell}=\sum\limits_{i=1}^{m}\sum\limits_{j=1}^{2p_{i}}\ub_{i,j}+\sum\limits_{i=1}^{n}\sum\limits_{j=1}^{2q_{i}}\vb_{i,j}+
\sum\limits_{i=1}^{k}\sum\limits_{j=1}^{2r_{i}}\wb_{i,j}+\wb+\ub+2{\ell}\vb.$$

For ${\ell}=1,\ldots,k$, let $$\gamma_{\ell}=\sum\limits_{i=1}^{m}\sum\limits_{j=1}^{2p_{i}}\ub_{i,j}+\sum\limits_{i=1}^{n}\sum\limits_{j=1}^{2q_{i}}\vb_{i,j}+
\sum\limits_{i=1}^{k}\sum\limits_{j=1}^{2r_{i}}\wb_{i,j}+\ub+\vb+2{\ell}\wb.$$

We now verify that $\alpha_{\ell}, \beta_{\ell}, \gamma_{\ell} \in \mathrm{relint} \RR_+(C)$ for all possible $\ell$.
For this, we put $u_i^{(1)}=\{u_{i,j}\mid  j=1,3,\ldots,2p_i-1\}$ for $i=1,\ldots, m$, $u_i^{(2)}=\{u_{i,j}\mid  j=2,4,\ldots,2p_i\}$ for $i=1,\ldots, m$ and $v_i^{(1)}, v_i^{(2)},w_i^{(1)}, w_i^{(2)}$ are defined similarly.

We see the following:

\begin{itemize}
\item Each of $u_{i,j}$'s, $v_{i,j}$'s and $w_{i,j}$'s is a regular vertex of $C$, while $u$, $v$ and $w$ are not.
\item An independent  subset $T$ of $V(C)$ is fundamental if and only if  $T$ is one of the following sets:
    \begin{enumerate}
      \item [$(i)$] $\bigcup\limits_{i=1}^{m}u^{(f_i)}_{i}$, where $(f_1,\ldots, f_m)\in \{1,2\}^m;$
      \vspace {1mm}

      \item [$(ii)$] $\bigcup\limits_{i=1}^{n}v^{(g_i)}_{i}$, where $(g_1,\ldots, g_n)\in \{1,2\}^n;$
      \vspace {1mm}

      \item [$(iii)$] $\bigcup\limits_{i=1}^{k}w^{(h_i)}_{i}$, where $(h_1,\ldots, h_k)\in \{1,2\}^k;$
      \vspace {1mm}

      \item [$(iv)$] $\{u\}\ \cup\ \bigcup\limits_{i=1}^{m}(u^{(f_i)}_{i}\setminus \{u_{i,1},u_{i,2p_i}\})\ \cup\ \bigcup\limits_{i=1}^{n}v^{(g_i)}_{i}\ \cup\ \bigcup\limits_{i=1}^{k}w^{(h_i)}_{i}\}$, where $(f_1,\ldots,f_m)\in \{1,2\}^m$, $(g_1,\ldots,g_n)\in \{1,2\}^n$ and  $(h_1,\ldots,h_k)\in \{1,2\}^k$;
          \vspace {1mm}

      \item [$(v)$] $\{v\}\cup\bigcup\limits_{i=1}^{n}(v^{(g_i)}_{i}\setminus \{v_{i,1},v_{i,2p_i}\})\cup\bigcup\limits_{i=1}^{m}u^{(f_i)}_{i}\cup\bigcup\limits_{i=1}^{k}w^{(h_i)}_{i}$ , where $(f_1,\ldots,f_m)\in \{1,2\}^m$, $(g_1,\ldots,g_n)\in \{1,2\}^n$ and  $(h_1,\ldots,h_k)\in \{1,2\}^k$;
          \vspace {1mm}

      \item [$(vi)$] $\{w\}\cup\bigcup\limits_{i=1}^{k}(w^{(h_i)}_{i}\setminus \{w_{i,1},w_{i,2p_i}\})\cup\bigcup\limits_{i=1}^{m}u^{(f_i)}_{i}\cup\bigcup\limits_{i=1}^{n}v^{(g_i)}_{i}$, where $(f_1,\ldots,f_m)\in \{1,2\}^m$, $(g_1,\ldots,g_n)\in \{1,2\}^n$ and  $(h_1,\ldots,h_k)\in \{1,2\}^k$.
    \end{enumerate}

\end{itemize}
It should be noted that there are $2^{m}$ fundamental sets in $(i)$ and $2^{m+n+k}$ fundamental sets in $(iv)$, and so on.
Hence, it follows from \eqref{eq:ineq} (see this in Subsection 1.3) that a vector of $\RR^{|V(C)|}$ of the form:
$$\sum\limits_{i=1}^{m}\sum\limits_{j=1}^{2p_{i}}{a_{i,j}\ub_{i,j}}+a\ub+
\sum\limits_{i=1}^{n}\sum\limits_{j=1}^{2q_{i}}{b_{i,j}\vb_{i,j}}+b\vb+\sum\limits_{i=1}^{k}\sum\limits_{j=1}^{2r_{i}}{c_{i,j}\wb_{i,j}}+c\wb$$ belongs to $ \RR_+(C)$ if and only if the following inequalities are satisfied:

\begin{enumerate}\label{inequality}
\item $a_{i,j} \geq 0$   for any $1\leq i\leq m$ and  $1\leq j\leq 2p_{i};$
\item $b_{i,j} \geq 0$   for any $1\leq i\leq n$ and  $1\leq j\leq 2q_{i};$
\item $c_{i,j} \geq 0$   for any $1\leq i\leq k$ and  $1\leq j\leq 2k_{i};$

\item $\sum\limits_{i=1}^{m}{\sum\limits_{j=1}^{2p_{i}}a_{i,j}}-\sum\limits_{u_{i,j}\in T}a_{i,j}+a \geq \sum\limits_{u_{i,j}\in T}a_{i,j}$ for any $T\in(i)$;
\vspace{1.5mm}

\item $\sum\limits_{i=1}^{n}{\sum\limits_{j=1}^{2q_{i}}b_{i,j}}-\sum\limits_{v_{i,j}\in T}b_{i,j}+b \geq \sum\limits_{v_{i,j}\in T}b_{i,j}$ for any $T\in(ii)$;
\vspace{1.5mm}

\item $\sum\limits_{i=1}^{k}{\sum\limits_{j=1}^{2r_{i}}c_{i,j}}-\sum\limits_{w_{i,j}\in T}c_{i,j}+c \geq \sum\limits_{w_{i,j}\in T}c_{i,j}$ for any $T\in(iii)$;
\vspace{1.5mm}

\item $\sum+b+c\geq a+2(\sum\limits_{u_{i,j}\in T}a_{i,j}+\sum\limits_{v_{i,j}\in T}b_{i,j}+\sum\limits_{w_{i,j}\in T}c_{i,j})$ for any $T\in(iv)$;
    \vspace{1.5mm}

\item $\sum+c+a\geq b+2(\sum\limits_{u_{i,j}\in T}a_{i,j}+\sum\limits_{v_{i,j}\in T}b_{i,j}+\sum\limits_{w_{i,j}\in T}c_{i,j})$ for any $T\in(v)$;
    \vspace{1.5mm}

\item $\sum+a+b\geq  c+2(\sum\limits_{u_{i,j}\in T}a_{i,j}+\sum\limits_{v_{i,j}\in T}b_{i,j}+\sum\limits_{w_{i,j}\in T}c_{i,j})$ for any $T\in(vi)$.
\end{enumerate}

Here, $\sum$ denotes  $\sum\limits_{i=1}^{m}{\sum\limits_{j=1}^{2p_{i}}a_{i,j}}+\sum\limits_{i=1}^{n}{\sum\limits_{j=1}^{2q_{i}}b_{i,j}}+\sum\limits_{i=1}^{k}{\sum\limits_{j=1}^{2r_{i}}c_{i,j}}$.
It is straightforward to check that $\alpha_{\ell}$ satisfies these inequalities, with strict inequalities holding for each $\alpha_{\ell}$. This implies that $\alpha_{\ell} \in \mathrm{relint}(\RR_+(C))\cap \mathbb{Z}^{|V(C)|}$.

Next, we show that $\alpha_{\ell}$ is a minimal vector  in $\mathrm{relint}(\RR_+(C))\cap \mathbb{Z}^{|V(C)|}$, i.e., it cannot be written
as a sum of an element in $\mathrm{relint}(\RR_+(C))\cap \mathbb{Z}^{|V(C)|}$ and an element in $\RR_+(C)\cap \mathbb{Z}^{|V(C)|} \setminus \{{\bf 0}\}$ for all $\ell=1,\ldots,m$.

Suppose on the contrary  that $\alpha_{\ell}=\alpha' + \alpha''$ for some $\alpha' \in \mathrm{relint}(\RR_+(C))\cap \mathbb{Z}^{|V(C)|}$ and $\alpha'' \in \RR_+(C)\cap \mathbb{Z}^{|V(C)|} \setminus \{{\bf 0}\}$.
Write
$$\alpha'=\sum\limits_{i=1}^{m}\sum\limits_{j=1}^{2p_{i}}{a'_{i,j}\ub_{i,j}}+a'\ub+
\sum\limits_{i=1}^{n}\sum\limits_{j=1}^{2q_{i}}{b'_{i,j}\vb_{i,j}}+b'\vb+\sum\limits_{i=1}^{k}\sum\limits_{j=1}^{2r_{i}}{c'_{i,j}\wb_{i,j}}+c'\wb,$$
$$\alpha''=\sum\limits_{i=1}^{m}\sum\limits_{j=1}^{2p_{i}}{a''_{i,j}\ub_{i,j}}+a''\ub+
\sum\limits_{i=1}^{n}\sum\limits_{j=1}^{2q_{i}}{b''_{i,j}\vb_{i,j}}+b''\vb+\sum\limits_{i=1}^{k}\sum\limits_{j=1}^{2r_{i}}{c''_{i,j}\wb_{i,j}}+c''\wb.$$

In view of the inequalities (1) - (3), we see that $a'_{i,j}, b'_{i,j}, c'_{i,j}\geq 1$ for all $i,j$. Because of the inequalities (4) - (6), we also see that $a', b', c'\geq 1$. Hence, $a''_{i,j}= b''_{i,j}= c''_{i,j}= b''= c''  = 0$ for all possible $i,j$.
From this together with (7) it follows that $a''\leq   0$.  Hence, $\alpha''=0.$ This is a contradiction, which shows that $\alpha_{\ell}$ is a minimal vector in $\mathrm{relint}(\RR_+(C))\cap \mathbb{Z}^{|V(C)|}$ for $\ell=1, \ldots, m$. Likewise,  so are $\beta_{\ell}$'s and $\gamma_{\ell}$'s.

\begin{Proposition} \label{three} Let $C$ be defined as before. Assume $m\leq n\leq k$. Then  $\mathrm{type} (\mathbb{K}[C])=m+n+k$,  and the top graded Betti numbers  of $\mathbb{K}[C]$ are given by $$\beta_{m+n+k,j}(\mathbb{K}[C])=\left\{
                                                        \begin{array}{ll}
                                                          1, & \hbox{$j=\mathrm{mat}(C)+n+m+\ell, \quad \ell=1,\ldots,k-n$;} \\
                                                           2, & \hbox{$j=\mathrm{mat}(C)+m+k+\ell, \quad \ell=1,\ldots,n-m$;} \\
                                                         3, & \hbox{$j=\mathrm{mat}(C)+k+n+\ell, \quad \ell=1,\ldots,m$;} \\
                                                          0, & \hbox{otherwise.}
                                                        \end{array}
                                                      \right.$$
                                                      \end{Proposition}
\begin{proof} Every minimal vector in $\mathrm{relint}(\RR_+(C))\cap \mathbb{Z}^{|V(C)|}$ corresponds to a minimal generator of $\omega_{\mathbb{K}[C]}$. It follows from the above discussion that $\mathrm{type} (\mathbb{K}[C])\geq m+n+k$.
 Since  $\mathrm{type} (\mathbb{K}[C])$ is equal to  the top total Betti number of $\mathbb{K}[C]$,
 we conclude that $\mathrm{type} (\mathbb{K}[C])\leq m+n+k$ by Theorem~\ref{main4}. Thus the first conclusion follows. From this, we see that $\alpha_{1},\ldots, \alpha_m, \beta_1,\ldots,\beta_n,\gamma_1,\ldots,\gamma_{k}$ are all the minimal vectors of $\mathrm{relint}(\RR_+(C))\cap \mathbb{Z}^{|V(C)|}$. Therefore, the set of  monomials  $$\{x^{\alpha_{\ell}},\ \ell=1,\ldots,m;\quad x^{\beta_{\ell}},\ \ell=1,\ldots, n;\quad  x^{\gamma_\ell},\ \ell=1,\ldots,k\}$$ is a minimal generating set of $\omega_{\mathbb{K}[C]}$, which is an ideal of the edge ring $\mathbb{K}[C]\subset \mathbb{K}[V(C)]$. Note that every monomial $x^\alpha$ belonging to $\mathbb{K}[C]$, which is  regarded as a graded  module over the standard graded ring $\mathbb{K}[E(C)]$,  has a degree of $\frac{1}{2}|\alpha|$.  Hence,
 $$\beta_{0,j}(\omega_{\mathbb{K}[C]})=\left\{
                                                        \begin{array}{ll}
                                                          3, & \hbox{$j=\sum\limits_{i=1}^mp_i+\sum\limits_{i=1}^nq_i+\sum\limits_{i=1}^kr_i+1+\ell, \ell=1,\ldots,m$;} \\
                                                           2, & \hbox{$j=\sum\limits_{i=1}^mp_i+\sum\limits_{i=1}^nq_i+\sum\limits_{i=1}^kr_i+1+\ell, \ell=m+1,\ldots,n$;} \\
                                                         1, & \hbox{$j=\sum\limits_{i=1}^mp_i+\sum\limits_{i=1}^nq_i+\sum\limits_{i=1}^kr_i+1+\ell, \ell=n+1,\ldots,k$;} \\
                                                          0, & \hbox{otherwise.}
                                                        \end{array}
                                                      \right.$$
                                                     Since  $|E(C)|=2(\sum\limits_{i=1}^mp_i+\sum\limits_{i=1}^nq_i+\sum\limits_{i=1}^kr_i)+m+n+k+3$ and $\mathrm{mat}(C)=\sum\limits_{i=1}^mp_i+\sum\limits_{i=1}^nq_i+\sum\limits_{i=1}^kr_i+1$, the second conclusion  follows by Lemma~\ref{can}.
\end{proof}

\subsection{type one} Let $A$ denote the compact graph $A_{\underline{p}}$, whose vertex set $V(A)$ and edge set $E(A)$ are given explicitly in Subsection~\ref{3.1}.  Then $|V(A)|=2\sum\limits_{i=1}^{m}p_i+1$. We may write $$\RR^{|V(A)|}=\{\sum\limits_{i=1}^{m}\sum\limits_{j=1}^{2p_{i}} a_{i,j}\ub_{i,j}+a\ub\mid\mbox{ all } a_{i,j},a\in \RR \}.$$ Here, $\ub_{i,j}, \ub$ correspond the vertices of $A$ in a natural way. Then, we could show the following vectors
 $$\alpha_{\ell}:=\sum\limits_{i=1}^{m}\sum\limits_{j=1}^{2p_{i}}\ub_{i,j}+2\ell \ub,\  \ell=1,\ldots,m-1$$
 are all the minimal vectors of $\mathrm{relint}(\RR_+(A))\cap \ZZ^{|V(A)|}$.

\begin{Proposition} \label{one} Let $A$ denote the compact graph  $A_{\underline{p}}$. Then  $\mathrm{type} (\mathbb{K}[A])=m-1$,  and the top graded Betti numbers  of $\mathbb{K}[A]$ are given by $$\beta_{m-1,j}(\mathbb{K}[A])=\left\{
                                                        \begin{array}{ll}
                                                          1, & \hbox{$j=\mathrm{mat}(A)+\ell,\ \ell=1,\ldots,m-1$;} \\

                                                          0, & \hbox{otherwise.}
                                                        \end{array}
                                                      \right.$$
                                                     \end{Proposition}

\subsection{type two}

Let $B^0$ and $B^s$ denote the compact graph $B_{\underline{p}:\underline{q}}^0$ and $B_{\underline{p}:\underline{q}}^s$ respectively. Their vertex sets $V(B^0)$ and $V(B^s)$ and edge sets $E(B^0)$ and $E(B^s)$ are given explicitly in Subsection~\ref{3.2}.  Then $|V(B^0)|=2(\sum\limits_{i=1}^{m}p_i+\sum\limits_{i=1}^{n}q_i)+2$. We may write $$\RR^{|V(B^0)|}=\{\sum\limits_{i=1}^{m}\sum\limits_{j=1}^{2p_{i}} a_{i,j}\ub_{i,j}+\sum\limits_{i=1}^{m}\sum\limits_{j=1}^{2p_{i}} b_{i,j}\vb_{i,j}+a\ub+b\vb\mid \mbox{ all } a_{i,j},b_{i,j}, a,b \in \RR \}.$$ Here, $\ub_{i,j},\vb_{i,j}, \ub, \vb$  correspond the vertices of $B^0$ in the natural way. Then, we could show the following vectors
 $$\alpha_{\ell}:=\sum\limits_{i=1}^{m}\sum\limits_{j=1}^{2p_{i}}\ub_{i,j}+\sum\limits_{i=1}^{n}\sum\limits_{j=1}^{2q_{i}}\vb_{i,j}+\vb+(2\ell+1)\ub, \quad \ell=0,\ldots,m-1$$
and $$\beta_{\ell}:=\sum\limits_{i=1}^{m}\sum\limits_{j=1}^{2p_{i}}\ub_{i,j}+\sum\limits_{i=1}^{n}\sum\limits_{j=1}^{2q_{i}}\vb_{i,j}+\ub+(2\ell+1)\vb, \quad \ell=1,\ldots,n-1$$
are all the minimal vectors of $\mathrm{relint}(\RR_+(B^0))\cap \ZZ^{|V(B^0)|}$.

On the other hand, we have $|V(B^s)|=|V(B^0)|+s-1$ and we may write
$$\RR^{|V(B^s)|}=\{\sum\limits_{i=1}^{m}\sum\limits_{j=1}^{2p_{i}}a_{i,j}\ub_{i,j}+\sum\limits_{i=1}^{n}\sum\limits_{j=1}^{2q_{i}}b_{i,j}\vb_{i,j}+
\sum\limits_{i=1}^{s-1}c_i\wb_{i}+a\ub+b\vb\mid \mbox{ all } a_{i,j},b_{i,j},c_i,a,b\in \RR\}.$$

Here, $\ub_{i,j},\vb_{i,j}, \wb_{i},\ub, \vb$  correspond the vertices of $B^s$ in the natural way. We may also show the following vectors
$$\alpha_{\ell}:=\sum\limits_{i=1}^{m}\sum\limits_{j=1}^{2p_{i}}\ub_{i,j}+\sum\limits_{i=1}^{n}\sum\limits_{j=1}^{2q_{i}}\vb_{i,j}+\sum\limits_{i=1}^{s-1}\wb_{i}+\vb+2\ell\ub, \ell=1,\ldots,m$$
and $$\beta_{\ell}:=\sum\limits_{i=1}^{m}\sum\limits_{j=1}^{2p_{i}}\ub_{i,j}+\sum\limits_{i=1}^{n}\sum\limits_{j=1}^{2q_{i}}\vb_{i,j}+\sum\limits_{i=1}^{s-1}\wb_{i}+\ub+2\ell\vb, \ell=1,\ldots,n$$
are all the minimal vectors of $\mathrm{relint}(\RR_+(B^s))\cap \ZZ^{|V(B^s)|}$.

\begin{Proposition} \label{two} Let $B^0$ and $B^s$ be defined as before. Assume $m\leq n$. Then the following statements hold:\begin{itemize}
                                                                     \item $\mathrm{type} (\mathbb{K}[B^0])=m+n-1$ and $\mathrm{type} (\mathbb{K}[B^s])=m+n$;
                                                                     \item the top graded Betti  numbers of $\mathbb{K}[B^0]$ are given by $$\beta_{m+n-1,j}(\mathbb{K}[B^0])= \left\{
                                                                                                          \begin{array}{ll}
                                                                                                            1, & \hbox{$j=\mathrm{mat}(B^0)+m-1+\ell, \quad \ell=1,\ldots,n-m$;} \\
                                                                                                            2, & \hbox{$j=\mathrm{mat}(B^0)+n-1+\ell, \quad \ell=1,\ldots,m-1$;} \\
                                                                                                            1, & \hbox{$j=\mathrm{mat}(B^0)+m+n-1$.}
                                                                                                            \\

                                                          0, & \hbox{otherwise.}
                                                                                                          \end{array}
                                                                                                        \right.$$

                                                                          \item the  top graded Betti  numbers of $\mathbb{K}[B^s]$ are given by $$\beta_{m+n,j}(\mathbb{K}[B^s])=\left\{
                                                                                                                   \begin{array}{ll}
                                                                                                                                                                                                                                          1, & \hbox{$j=\mathrm{mat}(B^s)+\ell, \quad \ell=m+1,\ldots,n$;} \\
                                                                                                                     2, & \hbox{$j=\mathrm{mat}(B^s)+n+\ell, \quad \ell=1,\ldots,m$;} \\

                                                          0, & \hbox{otherwise.}
                                                                                                                   \end{array}
                                                                                                                 \right.$$
                                                                   \end{itemize}

\end{Proposition}

\section{A question}

Let $G$ be a compact graph, and let $I_G$ be the toric ideal of $\mathbb{K}[G]$. Denote by $J_G$ the initial ideal of $I_G$ with respect to the order given in Section 3. As we have seen in the previous section,  $I_G$ and $J_G$ share the same top graded Betti numbers.   This naturally leads to the following question:
\begin{center}
{\it Does  $I_G$ and $J_G$ always share the same  graded Betti numbers?}
\end{center}
Unfortunately, we are unable to provide a general answer to this question, except for a very specific case when $G$ is a compact graph of type one.

 In what follows, we use $A$ to denote  the compact graph $A_{\underline{p}}$, where $\underline{p}=(\overbrace{p,\ldots,p}^{m})$ is a vector in $\ZZ_+^m$.   Let $f(t)$ and $g(t)$ denote the polynomial $\sum\limits_{i,j}\beta_{i,j}(I_A)(-1)^it^j$ and  $\sum\limits_{i,j}\beta_{i,j}(J_A)(-1)^it^j$, respectively. It is known $f(t)=g(t)$ and $\beta_{i,j}(I_A)\leq \beta_{i,j}(J_A)$ for all $i,j$.

\begin{Proposition} If $2\leq m\leq p+3$, then $$\beta_{i,j}(I_A)=\beta_{i,j}(J_A)=\left\{
                                                                   \begin{array}{ll}
                                                                     \binom{m}{i+2}, & \hbox{$j=(i+2)p+\ell, \ell=1,\ldots,i+1$;} \\
                                                                     0, & \hbox{otherwise.}
                                                                   \end{array}
                                                                 \right.
$$
\end{Proposition}

\begin{proof} Put $A_i=\{j\in \mathbb{Z}\mid \beta_{i,j}(J_A)\neq 0\}$ for all $i\geq 0$. Then, by Proposition~\ref{Betti}, we have $A_i=\{(i+2)p+\ell\mid \ell=1,\ldots,i+1\}$ for $0\leq i\leq m-2$,  and is $\emptyset$ otherwise. Given that $j\notin A_i$ it can be inferred that $\beta_{i,j}(J_A)=\beta_{i,j}(I_A)=0$. Therefore, we will next consider only the case when $j\in A_i$.

(1) If $m\leq p+2$ then it follows that $A_{i_1}\cap A_{i_2}=\emptyset$ for any distinct  $i_1$ and $i_2$.
  Consequently,  for any $j\in A_i$, the coefficient of $t^j$ in $f(t)$  is $(-1)^i\beta_{i,j}(J_A)$, while in $g(t)$ it is  $(-1)^i\beta_{i,j}(I_A)$. Therefore we can deduce that  $\beta_{i,j}(J_A)=\beta_{i,j}(I_A)$.

(2) If $m=p+3$, then for any pair $i_1\neq i_2$, $A_{i_1}\cap A_{i_2}\neq \emptyset$ if and only if $\{i_1,i_2\}=\{m-3,m-2\}$ and in that case $A_{m-3}\cap A_{m-2}=\{mp+1\}=\{(m-1)p+m-2\}.$ If $j\neq mp+1$ then it follows that $\beta_{i,j}(I_A)=\beta_{i,j}(J_A)$ for the same reason as in (1). If $j=mp+1$ then, by comparing the coefficients of $t^{mp+1}$ in polynomials $f(t)$ and $g(t)$, we conclude that $$\beta_{m-3,mp+1}(I_A)-\beta_{m-2, mp+1}(I_A)=\beta_{m-3,mp+1}(J_A)-\beta_{m-2, mp+1}(J_A).$$  On the other hand, we have $\beta_{m-2, mp+1}(I_A)=\beta_{m-2, mp+1}(J_A)$ by Theorem~\ref{5.1}. From this it follows that  $\beta_{m-3,mp+1}(I_A)=\beta_{m-3,mp+1}(J_A)$, as required.
\end{proof}

\vspace{4mm}

{\bf \noindent Acknowledgment:} This project is supported by NSFC (No. 11971338) The authors are grateful to the software systems \cite{C1}
for providing us with a large number of examples to develop ideas and test our results.

\vspace{5mm}

{\bf \noindent Data availability:} The data used to support the findings of this study are included within the article.

\vspace{5mm}

{\bf \noindent Statement:}    On behalf of all authors, the corresponding author states that there is no conflict of interest.

\end{document}